\documentclass[12pt,a4paper,oneside]{article}
\usepackage{amsmath,amstext,amssymb,amsopn,amsthm,color}
\usepackage{graphicx}
\usepackage{hyperref}

\usepackage{ushort}

\definecolor{pk}{rgb}{0.1,0.5,0.1}
\definecolor{kk}{rgb}{0.1,0.2,0.7}
\definecolor{tg}{rgb}{0.7,0.1,0.2}

\newcommand{\kk}[1]{{\textcolor{kk}{#1}}}

\newcommand{\gener}{{\cal L}}

\theoremstyle{plain}

 \newtheorem{thm}{Theorem}[section] 
\newtheorem{lem}[thm]{Lemma} \newtheorem{prop}[thm]{Proposition}
\newtheorem{cor}[thm]{Corollary}
 \theoremstyle{definition}
\newtheorem{definition}[thm]{Definition}

 \theoremstyle{remark}
\newtheorem*{rem*}{Remark} 
\newtheorem{remark}[thm]{Remark}

\numberwithin{equation}{section}

\newcommand{\nn}{\nonumber}

\newcommand{\R}{\mathbb{R}}

 \newcommand{\Rd}{{\R^{d}}}

\newcommand{\N}{\mathbb{N}}

\renewcommand{\H}{\mathbb{H}} 
\renewcommand{\leq}{\leqslant} \renewcommand{\le}{\leq}
\renewcommand{\geq}{\geqslant} \renewcommand{\ge}{\geq}

\DeclareMathOperator{\dist}{dist}
\DeclareMathOperator{\diam}{diam}

\DeclareMathOperator{\supp}{supp}   
 \def\({\left(} \def\){\right)} \def\[{\left[}
  \def\]{\right]} \def\<{\langle} \def\>{\rangle}

\newcommand{\E}{\mathbb{E}}
\newcommand{\p}{\mathbb{P}}

\ushortCreate*(3)(.46){uline}

\newcommand{\wt}{\widetilde}
\newcommand{\wh}{\widehat}
\newcommand{\ol}{\overline}
\newcommand{\eps}{\varepsilon}

\newcommand{\WUSC}[3]{\textrm{WUSC}(#1,#2,#3)}
\newcommand{\WLSC}[3]{\textrm{WLSC}(#1,#2,#3)}
\newcommand{\lC}{{\underline{c}}}
\newcommand{\uC}{{\overline{C}}}
\newcommand{\la}{{\underline{\alpha}}}
\newcommand{\ua}{{\overline{\alpha}}}

\title{Estimates of Dirichlet heat kernel for symmetric Markov processes
\author{Tomasz Grzywny\thanks{The first author was partially supported by  National Science Centre (Poland): grant 2014/14/M/ST1/00600 and the Alexander von Humboldt Foundation},
Kyung-Youn Kim\thanks{The second author was partially supported by DAAD PaJaKo Programme and the German Research Foundation (DFG) through the International Research Training Group (IRTG) 2235.},
Panki Kim\thanks{This work was  supported by the National Research Foundation of 
Korea(NRF) grant funded by the Korea government(MSIP) (No. NRF-2015R1A4A1041675)} }
 }

\begin{document}
\maketitle
\begin{abstract}
We consider a large class of symmetric pure jump Markov processes dominated by isotropic unimodal L\'evy processes with weak scaling conditions. First, we establish sharp two-sided heat kernel estimates for these processes in  $C^{1,1}$ open sets.
As corollaries of our main results, we obtain sharp two-sided Green function estimates and a scale invariant boundary Harnack inequality with explicit decay rates in $C^{1,1}$ open sets.
\end{abstract}

\section{Introduction}

The study of the heat kernel of a semigroup  is a field of interactions between probability, analysis and geometry. 
The transition density function provides direct access to  the path properties of  a  Markov process. 
In addition, it is the fundamental solution(or heat kernel) of the heat equation with an infinitesimal generator of the corresponding process. 
The Dirichlet heat kernel describes an operator with zero exterior conditions. For instance, the Green function and the solutions to Cauchy and Poisson problems with Dirichlet conditions are expressed by the heat kernel. 
In this paper, we consider  a large class of symmetric pure jump Markov processes dominated by isotropic unimodal L\'evy processes with weak scaling conditions.
We estimate the transition density $p_D(t,x,y)$ of such Markov processes killed upon leaving an open set $D \subset \Rd$ with $C^{1,1}$ smoothness of the boundary.                    
In other words, we establish  sharp two--sided estimates of the Dirichlet heat kernel of the {integro-differential} operators with the maximum principle. Such operators
are commonly used to model nonlocal phenomena \cite{ MR2373320, MR2680400,  MR2583992, MR3034006,MR3116931, MR2042661}

The Dirichlet heat kernel estimates of the Laplacian (the Brownian motion) on the bounded $C^{1,1}$ domains were obtained  in \cite{Davies, Hui} (the upper bound) and
\cite{ MR1900329} (the lower bound). See \cite{Cho} for the Dirichlet heat kernel estimates for more general diffusions, and see \cite{MR1969798} for bounds of the Dirichlet heat
kernel of  the Laplacian on the bounded Lipschitz domain.

For the fractional Laplacian, in 2010, Chen et al \cite{MR2677618} gave sharp (two-sided) explicit estimates for the Dirichlet heat kernel $p_{D}(t, x, y)$ of the fractional Laplacian in  any $C^{1, 1}$ open set $D$ and over any finite time interval (see \cite{MR2722789}  for an extension to non-smooth open sets).
When $D$ is bounded, large--time Dirichlet heat kernel estimates can be deduced easily from short--time estimates using a spectral analysis.

The approach developed in \cite{MR2677618} provides a road map for establishing sharp two-sided Dirichlet heat kernel estimates of other discontinuous processes, and the result has been generalized to more general stochastic processes:
purely discontinuous symmetric L\'evy processes (\cite{MR3237737, BGR2013_3}), 
symmetric L\'evy processes with Gaussian component (\cite{CKS1}),
 symmetric non-L\'evy processes (\cite{MR2574732, KK}) and 
non-symmetric stable processes with gradient perturbation (\cite{MR3050510}). 

Let $\p_y(\tau_{D}>t)$ be the survival probability of the corresponding process and $p(t,x,y)=p_\Rd(t,x,y)$ be the (free) heat kernel for $D=\Rd$. 
Another form of  two-sided heat kernel estimates is the following factorization;
\begin{equation}\label{eq:ehks}
c_1  \p_x(\tau_{D}>t)
      \p_y(\tau_{D}>t)p(t,x,y) \le    {p_{D}(t, x, y)} \le c_2   \p_x(\tau_{D}>t)
      \p_y(\tau_{D}>t)p(t,x,y).
  \end{equation}
In fact, \eqref{eq:ehks} holds for more general sets such as the Lipschitz open set. See \cite{MR2722789, MR3237737}. See \cite{BGR1} for a direct approach to obtain sharp estimates of the survival probabilities of unimodal L\'evy processes.

Extensions of the result in \cite{MR2677618} were obtained for a quite large class of symmetric L\'evy processes, including general unimodal L\'evy processes with L\'evy densities satisfying weak scaling conditions, in \cite{MR3237737, BGR2013_3}.
However, the extension to symmetric Markov processes with jumping kernels satisfying similar weak scaling conditions is unknown.
 In this paper, we extend the results of \cite{BGR2013_3} and \cite{KK} to 
 more general processes that are non-isotropic and non-L\'{e}vy. Our results cover not only a large class of symmetric Markov processes with jumping kernels satisfying weak scaling conditions, but  also a large class of symmetric Markov processes with jumping kernels that decay exponentially with the damping exponent $\beta \in (0, \infty)$, and symmetric finite range Markov processes.
 
\medskip

For two non-negative functions $f$ and $g$, the notation $f\asymp g$ means that there are positive constants $c_1$ and $c_2$ such that $c_1g(x)\leq f (x)\leq c_2 g(x)$ in the common domain of the definition of $f$ and $g$.
We use the symbol ``$:=$,'' which is read as ``is defined to be.''
For $a, b\in \R$, $a\wedge b:=\min \{a, b\}$ and $a\vee b:=\max\{a, b\}$.
We use $dx$ to denote the Lebesgue measure in $\R^d$. For a Borel set $A\subset \R^d$, we use $|A|$ to denote its Lebesgue measure.

\medskip

For $0<\la\le \ua<2$, let $\phi$ be an increasing function on $[0, \infty)$ satisfying that there exist positive constants $\lC\le 1$ and $1\le \uC$ such that 
\begin{description}
\item[{\bf(WS)}]
$$\lC\left(\frac{R}{r}\right)^{\la}\le \frac{\phi(R)}{\phi(r)}\le \uC \left(\frac{R}{r}\right)^{\ua} \qquad \mbox{ for }\,\, 0<r\, \le  R. $$
\end{description}
Using this $\phi$,  we define
\begin{align}\label{d:nu}
\nu(r):=\frac1{\phi(r)r^d}\qquad \mbox{ for}\,\, r>0.
\end{align} 
Note that according to {\bf(WS)} and \eqref{d:nu}, there exists $c=c(\ua,  \uC, d)$ such that 
\begin{align}
\nu(r)\le c\, \nu(2r)\qquad\mbox{ for any }\,\,r>0.  \label{nu1}
\end{align}

A measure on $\Rd$ is called {\it isotropic unimodal}, if it is absolutely continuous on $\Rd\backslash\{0\}$
with a radial non-increasing density function and
a L\'evy process is called {\it isotropic unimodal}, if the {one-dimensional} distributions are unimodal. 
Since {\bf(WS)} implies 
$$\int_{\R^d}(1\wedge |x|^2)\nu(|x|)dx\le c \left(\int_0^{1}s^{-\ua+1} ds+ \int_{1}^{\infty} s^{-\la-1} \right)ds<\infty,$$
$\nu(dx):=\nu(|x|)dx$ is unimodal L\'{e}vy measure, and by Proposition in \cite{MR705619} there exists  a pure jump isotropic unimodal L\'{e}vy process $Z$  corresponding to $\nu$.

Throughout this paper, we assume that 
$\kappa:\Rd\times\Rd\to (0,\infty)$  is a symmetric measurable function and 
 there exists  ${L_0}>1$ such that 
\begin{align}\label{a:kappa}
L_0^{-1}\leq \kappa(x,y)\leq L_0,\quad
x,\,y\in\Rd.
\end{align}

Let $J:\Rd\times\Rd\to (0,\infty)$ be a symmetric measurable function, which is the jumping kernel of our process. We consider two sets of conditions on $J$. The first set is as follows:
\begin{description}
\item[{(J1)}] 
$ 
\mbox{\bf (J1.1)}\,\,  J(x, y) =  \kappa(x,y)   \nu(|x-y|)\mbox{ on } |x-y|\le 1,\\ 
\mbox{\bf (J1.2)}\,\, \sup_{x\in \R^d}\int_{|x-y|>1} J(x, y)dy<\infty, \\
\mbox{\bf (J1.3)}\,\, \mbox{ For any } M>0, \mbox{ there exists } C_M>1 \mbox{ such that } C_M^{-1} \nu(|x-y|) \le J(x, y)\le C_M \nu(|x-y|) \mbox{ for } |x-y|<M.
$
\end{description}
The constant $1$ in the condition {\bf (J1.1)} plays no special role, and it can be changed 
to any small positive real number.
Since {\bf (J1.3)} implies {\bf (J1.1)}, $J$ satisfies conditions {\bf (J1.2)} and {\bf (J1.3)} if and only if $J$ satisfies condition {\bf (J1)}.

For the second set of conditions on $J$, let 
 $\chi$ be a non-decreasing function on $(0,\infty)$ with $\chi(r)\equiv \chi(0)$, $r\in(0,1]$, and there exists $\gamma_1,\,\gamma_2,\,L_1\,L_2>0$ and $\beta \in[0,\infty]$ such
that
\begin{equation}\label{e:Exp}
 L_1e^{\gamma_1r^{\beta}}\leq \chi(r)\leq  L_2e^{\gamma_2r^{\beta}},\quad r>1.
\end{equation}
Then, the second condition on $J$ is as follows:
\begin{description}
\item[(J2)]
$  J(x, y) =  \kappa(x,y)   \nu(|x-y|)\chi(|x-y|)^{-1},\quad {x,y\in\Rd}$,
\end{description} which is equal to 
$$ \begin{cases}
 \kappa(x,y) \left(\phi(|x-y|)|x-y|^{d} \cdot \chi (|x-y|)\right)^{-1}  & \text{ if  } \beta \in [0, \infty),\\
  \kappa(x,y) \left(\phi(|x-y|) |x-y|^{d}\right)^{-1}  {\bf 1}_{\{|x-y| \le1\}} & \text{ if  } \beta = \infty.
  \end{cases}
$$
Clearly {\bf(J2)} implies {\bf(J1.1)} and {\bf(J1.2)}. Moreover,  if  {\bf(J2)} holds and $\beta\not= \infty$, then {\bf(J1)} holds.

\medskip

We consider the Dirichlet form $(\mathcal{E},\mathcal{F})$ associated with the jumping kernel $J$:
$$\mathcal{E}(u,v):=\frac{1}{2}\int\int(u(x)-u(y))(v(x)-v(y){)}J(x,y)dxdy,$$
and  
$\mathcal{F}:=\{u\in L^2(\Rd): \mathcal{E}(u,u)<\infty\}$.
Under conditions  {\bf (J1.1)} and {\bf (J1.2)}, according to \cite[Theorem 2.1]{MR2365348} and \cite[Theorem 2.4]{MR2888033},
$(\mathcal{E},\mathcal{F})$ is  a regular (symmetric) Dirichlet form on $L^2(\R^d,dx)$.  Moreover, the corresponding Hunt process $Y$ is conservative  and $Y$ has the H\"{o}lder continuous transition density $p(t,x,y)$ on $(0,\infty)\times \Rd\times\Rd$ (see \cite{MR2524930}).

\medskip

Now, we state the estimates of the transition density $p(t,x,y)$ of $Y$ with the jumping intensity kernel $J$ {satisfying} either the conditions {\bf (J1.2)} and {\bf (J1.3)}, or the condition {\bf(J2)}. 
The proof of the upper bound of Theorem \ref{T:n2.1} is almost the same as that of \cite[(2.6)]{MR3237737} using the condition {\bf (J1.3)} instead of that of \cite[(1.5)]{MR3237737}.
Thus, we skip the proof of the upper bound. The proof of the lower bound is given in Section \ref{sec:plbd}.

\begin{thm}\label{T:n2.1}
Suppose that $Y$ is a symmetric pure jump Hunt process whose jumping intensity kernel $J$ satisfies the conditions {\bf (J1.2)} and {\bf (J1.3)}. 
Then, for each $M>0$ and $T>0$, 
there is a positive constant $ C_{\ref{T:n2.1}}\ge 1$ that depends on $\phi, L_0, M$ and $T$ such that for every $(t, x, y)\in (0,T]\times \Rd\times\Rd$ with $|x-y|< M$,  the function $p(t, x, y)$ has the following estimates:
\begin{align*}
 C_{\ref{T:n2.1}}^{-1}\left([\phi^{-1}(t)]^{-d}\wedge  t\nu(|x-y|)\right)\,\le\,
p(t,x,y)\,\le\, 
C_{\ref{T:n2.1}}\,\left([\phi^{-1}(t)]^{-d}\wedge  t\nu(|x-y|)\right)
\end{align*}
where $\phi^{-1}(t)$ is the inverse function of $\phi(t)$.
\end{thm}

For notational convenience, for each  $a, \gamma, T>0$, we define a function $F_{a, \gamma, T}(t, r)$ on $(0, T] \times [0, \infty)$ as
\begin{equation}\label{eq:qd}
F_{a, \gamma, T}(t, r):
=\begin{cases}
[\phi^{-1}(t)]^{-d}\wedge  t\nu(r)e^{-\gamma r^\beta}&\text{ if } \beta\in[0,1],\\
[\phi^{-1}(t)]^{-d}\wedge  t\nu(r)&\text{ if } \beta\in (1, \infty] \text{ with } r<1,\\
t \exp\left\{-a \left( r \, \left(\log \frac{ T r}{t}\right)^{\frac{\beta-1}{\beta}}\wedge  r^\beta\right) \right\}\qquad &\hbox{ if } \beta\in(1, \infty) \text{ with } r \ge 1,\\
\left(t/(T r)\right)^{ar}=\exp\left\{-ar\left(\log \frac{Tr}{t}\right)\right\} 
&\hbox{ if }  \beta=\infty \text{ with } r \ge 1.
\end{cases}
\end{equation}

\begin{thm}\label{T:2.1}
Suppose that $Y$ is a symmetric pure jump Hunt process whose jumping intensity kernel $J$ satisfies the condition {\bf (J2)}. Then, the process $Y$ has a continuous transition density function $p(t, x, y)$ on $(0, \infty)\times \R^d\times \R^d$.
For each $T>0$, there are positive constants $C_{\ref{T:2.1}}\ge 1$, $c_1$ and $c_2\ge 1$ that depend on $\phi,  L_0, \beta,\chi$ and $T$ such that for every $t \in (0,T]$ the function $p(t, x, y)$ has the following estimates:
\begin{align*}
 c_2^{-1}F_{c_1, \gamma_2, T}(t, |x-y|)\,\le\,
p(t,x,y)\,\le\, 
c_2\,F_{C_{\ref{T:2.1}}, \gamma_1, T}(t, |x-y|).
\end{align*}
\end{thm}

Theorem \ref{T:2.1} for the case $\beta\in(1, \infty]$ is basically derived from \cite[Theorems 1.2 and 1.4]{MR2806700}. Despite using $1$ instead of $T$, the proof is the same.
When $\beta \in [0,1]$, the upper bound in Theorem \ref{T:2.1} is derived from \cite[Theorem 2, Proposition 1]{KS}.
The lower bound  in Theorem \ref{T:2.1} is proved as a special case of the preliminary lower bound on the heat kernel of the killed process  in Section~\ref{sec:plbd}.

\medskip

For any open set $D\subset \R^d$, the {\it first exit time of $D$} by process $Y$ is defined by the formula 
$\tau_D:= \inf\{t>0: Y_t\notin D\}$
and we use $Y^D$ to denote the process obtained by killing process $Y$ upon exiting $D$.
The strong Markov property is used to easily verify that
$p_D(t, x, y):=p(t, x, y)-\E_x[ p(t-\tau_D, Y_{\tau_D}, y); t>\tau_D]$
is the transition density of $Y^D$.
Using the continuity and estimate of $p$, it is routine to show that $p_D(t, x, y)$ is symmetric and continuous (e.g., see the proof of Theorem~2.4 in~\cite{MR1329992}).

Let $D\subset \R^d$ (when $d\ge 2$)  be a $C^{1,1}$ open set, that is, there exists a localization radius $ R_0>0 $ and a constant $\Lambda>0$ such that for every $z\in\partial D$ there exists a $C^{1,1}$-function $\varphi=\varphi_z: \R^{d-1}\to \R$ satisfying $\varphi(0)=0$, $\nabla\varphi (0)=(0, \dots, 0)$, $\| \nabla\varphi \|_\infty \leq \Lambda$, $| \nabla \varphi(x)-\nabla \varphi(w)| \leq \Lambda |x-w|$ and an orthonormal coordinate system $CS_z$ of  $z=(z_1, \cdots, z_{d-1}, z_d):=(\wt z, \, z_d)$ with  an origin at $z$ such that $ D\cap B(z, R_0 )= \{y=({\tilde y}, y_d) \in B(0, R_0) \mbox{ in } CS_z: y_d > \varphi (\wt y) \}$. 
The pair $( R_0, \Lambda)$ is called the $C^{1,1}$ characteristics of the open set $D$.
Note that a $C^{1,1}$ open set $D$ with characteristics $(R_0, \Lambda)$ can be unbounded and disconnected, and the distance between two distinct components of $D$ is at least $R_0$.
By a $C^{1,1}$ open set  in $\R$ with a characteristic $R_0>0$, we mean an open set that can be written as the union of disjoint intervals so that the {infimum} of the lengths of all these intervals is {at least $R_0$} and the {infimum} of the distances between these intervals is  {at least $R_0$}.

 It is well known that if $D$ is $C^{1,1}$ open set with the characteristics $(R_0, \Lambda)$, then $D$ satisfies the interior and exterior ball conditions with the characteristic $R_1\le R_0$. That is, there exist balls $B_1,B_2\subset \Rd$ with radius $R_1$ such that $B_1\subset D\subset B_2^c$ satisfying $\delta_{B_1}(x)=\delta_{D}(x)=\delta_{B_2}(x)$ for any $x\in B_1$. Throughout this paper, without loss of generality, we always  assume that $R_0=R_1$.

To obtain the sharp estimates of the exit distributions for $Y$ (see Theorem \ref{T:exit}), we need additional conditions for the regularity of $\kappa$ and $\phi$.
\begin{description}
\item[(K$_\eta$)]
There are $L_3>0$ and $\eta> \ua/2$ such that $|\kappa(x,x+h)-\kappa(x,x)|\leq L_3|h|^{\eta}$ for every $x,\,h\in \Rd$, $|h|\leq 1$.
\end{description}
Note that the condition {\bf(K$_\eta$)} implies that 
\begin{align*}
|\kappa(x+h_1, x+h_2)-\kappa(x, x)|\le 2L_3(|h_1|^{\eta}+|h_2|^{\eta}), \qquad \mbox{for } |h_1|, |h_2|<1. 
\end{align*}
\begin{description}
\item[{\bf(SD)}]
$\phi\in C^1(0,\infty) $ and $ r\to -\nu'(r)/r$  is decreasing.
\end{description}
(See Remark \ref{r:sd}.)

Now, we state the following theorem, which is one of the main results of this paper.
Let $\delta_D(x)$ be the distance between $x$ and $D^c$, and let
\begin{align}\label{e:dax}
\Psi(t,x):=\left(1\wedge\sqrt{\frac{\phi(\delta_D(x))}{t}}\right).
\end{align}
\begin{thm}\label{t:nmain}
Suppose that 
$Y$ is a symmetric pure jump Hunt process whose jumping intensity kernel $J$ satisfies the conditions {\bf (J1)}, {\bf (SD)}, and {\bf (K$_\eta$)}. 
Suppose that  
$D$ is a bounded $C^{1,1}$ open set in $\R^d$ with characteristics $(R_0, \Lambda)$.
{Then, for each $T>0 $, there exists $c_1=c_1(\phi, L_0, L_3,  \eta, R_0, \Lambda,  T, d)$  and $c_2=c_2(\phi, L_0, L_3,  \eta, R_0, \Lambda,  T, d, \text{\rm diam}(D))>0$} such that the transition density $p_D(t,x,y)$ of $Y^D$ has the following estimates.
\begin{description}
\item{\rm (1)}
For any $(t, x, y)\in (0, T]\times D\times D$,
we have 
\begin{align*}
c_1^{-1} \Psi(t,x) \Psi(t,y)\, p(t, x, y)\le  p_D(t, x, y)\le c_1 \Psi(t,x) \Psi(t,y) \,p(t, x, y)
\end{align*}
\item{\rm (2)}
For any  $(t, x, y)\in [T, \infty)\times D\times D$, we have
\begin{align*}
{c_2^{-1} e^{- t\, \lambda^{D}}\, \sqrt{\phi(\delta_D(x))}\, \sqrt{\phi(\delta_D(y))}}
 \le p_D(t, x, y) \le c_2 e^{- t\, \lambda^{D}}\, \sqrt{\phi(\delta_D(x))}\, \sqrt{\phi(\delta_D(y))},
\end{align*}
where $-\lambda^D<0$ is the largest eigenvalue of the generator of $Y^D$.
\end{description}
\end{thm}

\medskip

\begin{remark}\label{r:sd}
Conditions{ \bf(SD)} and {\bf(WS)} hold for a large class of  pure jump isotropic unimodal L\'{e}vy processes  including all  {subordinate} Brownian motions with weak scaling conditions (see \eqref{e:wecsb}):
let  $W=(W_t, \p_x)$ be a Brownian motion in $\R^d$ and $S=(S_t)$  {be} an independent driftless subordinator with Laplace exponent $\varphi_1$. 
The Laplace exponent   $\varphi_1$ is a  Bernstein function with $\varphi_1(0+)=0$.  Since $\varphi_1$ has no drift part,  $\varphi_1$ can be written in the form $$
\varphi_1(\lambda)=\int_0^{\infty}(1-e^{-\lambda t})\,
\mu(dt)\, .
$$
Here  $\mu$ is a $\sigma$-finite measure on
$(0,\infty)$ satisfying
$
\int_0^{\infty} (t\wedge 1)\, \mu(dt)< \infty.
$
 $\mu$ is called the L\'evy measure
of the subordinator $S$.

The subordinate Brownian motion $Z=(Z_t, \p_x)$ is defined by $Z_t=W_{S_t}$. 
The density of the  L\'evy measure of $Z$  with respect to the Lebesgue measure is given by $x \to \nu_d(|x|)$ with
$$
\nu_d(r)=\int^\infty_0 (4\pi t)^{-d/2}\exp\left(-\frac{r^2}{4t}\right) \mu(t)dt, \qquad r\neq 0. 
$$
Thus, $r \to \nu_{d} (r)$ is smooth for $r>0$, and
$$
-\frac{\nu_{d}'(r)}{r}=2 \pi \nu_{d+2}(r), \quad r>0
$$
which is decreasing. 
Suppose that 
\begin{align}
\label{e:wecsb}
\lC\left(\frac{R}{r}\right)^{\la/2}\le \frac{\varphi_1(R)}{\varphi_1(r)}\le \uC \left(\frac{R}{r}\right)^{\ua/2} \qquad \mbox{ for }\,\, 0<r\, \le  R. 
\end{align}
Then, by   \cite[Theorem 26]{2013arXiv1305.0976B}
$$\nu_d(r) \asymp r^{-d} \varphi_1(r^{-2}).$$
 Let $\wh \phi(r):= r^{-d} \nu_d(r)^{-1}$ (so that $\nu_d(r)=\wh \phi(r)^{-1}r^{-d}$). Then $\wh \phi$ is smooth, and since 
$\wh \phi(r) \asymp \varphi_1(r^{-2})^{-1}$, it satisfies  
{\bf(WS)}.
\end{remark}

When either $D$ is unbounded or $\beta=\infty$, we need precise information on $J$,  which is encoded in {\bf(J2)}, for large $|x-y|$. 
Moreover, when $\beta \in (1, \infty]$, we need to impose  an addition assumption for $D$ in order to obtain the sharp lower bound of $p_D(t,x,y)$;
We say that {\it the path distance in an open set $U$ is comparable to the Euclidean distance with characteristic $\lambda_1$} if for every $x$ and $y$ in  $U$ there is a rectifiable curve $l$ in $U$ that  connects $x$ to $y$ such that the length of $l$ is less than or equal to  $\lambda_1|x-y|$.
Clearly, such a property holds for all bounded $C^{1,1}$ connected open sets, $C^{1,1}$ connected open sets with compact complements, and connected open sets above graphs of $C^{1,1}$ functions.

\begin{thm}\label{t:main}
Suppose that $Y$ is a symmetric pure jump Hunt process whose jumping intensity kernel $J$ satisfies the conditions {\bf (J2)}, {\bf(SD)} and {\bf(K$_\eta$)}.
Suppose that $D$ is a $C^{1,1}$ open set in $\R^d$ with characteristics $( R_0, \Lambda)$.
Then, for each $T>0$, the transition density $p_D(t,x,y)$ of $Y^D$ has the following estimates.
\begin{description}
\item{\rm (1)}
There is a positive constant $c_1=c_1(\beta, \chi, \phi, L_0, L_3,  \eta, R_0, \Lambda,  T,  d)$  such that for all $(t, x, y)\in (0, T]\times D\times D$ we have 
\begin{align*}
 p_D(t, x, y)
\leq \,c_1 \, \Psi(t,x) \Psi(t,y)
\begin{cases}
F_{ C_{\ref{T:2.1}}\wedge\gamma_1, \gamma_1, T}(t, |x-y|/{6})&\mbox{ if } \beta\in[0,\infty),\\
F_{ C_{\ref{T:2.1}},  \gamma_1, T}(t, |x-y|/{4})&\mbox{ if } \beta=\infty,
\end{cases}
\end{align*}
where $C_{\ref{T:2.1}}$ is the constant in Theorem~\ref{T:2.1}.

\item{\rm (2)}
There is a positive constant $c_2=c_2( \beta, \chi, \phi, L_0, L_3,  \eta, R_0, \Lambda,  T, d)$  such that for all $t \in (0, T]$ we have
\begin{align*}
p_D(t, x, y)\ge \,c_2 \,&\Psi(t,x)\Psi(t,y) 
\begin{cases}
[\phi^{-1}(t)]^{-d}\wedge {te^{-\gamma_2 |x-y|^\beta}}\nu(|x-y|)&
\hbox{if } \beta\in[0,1],\\
[\phi^{-1}(t)]^{-d}\wedge{t}\nu(|x-y|)& \hskip -.1in
\begin{array}{c}
\hbox{if  $\beta\in(1, \infty)$ and $|x-y|<1$},\\
\hbox{or  $\beta=\infty$ and $|x-y| \le 4/5$}.
\end{array}
\end{cases}
\end{align*}
\item{\rm (3)}
Suppose, in addition, that the path distance in $D$ is comparable to the Euclidean distance with characteristic $\lambda_1$. 
Then, there are positive constants $c_i=c_i( \beta, \chi, \phi, L_0, L_3,  \eta, R_0, \Lambda,  T, d,  \lambda_1)$, $i=3,4$, such that if $x, y \in D$  and $t \in (0, T]$, we have
\begin{align*}
p_D(t, x, y) \ge \,
c_3 \,\Psi(t,x) \Psi(t,y)  
\begin{cases}
F_{ c_4,  \gamma_2, T}(t, |x-y|)&\hbox{if $\beta\in (1,\infty)$ and $|x-y|\ge 1$},\\
F_{ c_4,  \gamma_2, T}(t, 5|x-y|/4)&\hbox{if    $\beta=\infty$ and $|x-y| \ge 4/5$}.
\end{cases}
\end{align*}
\item{\rm (4)}
If $\beta\in (1, \infty)$, there is a positive constant $c_5=c_5(\beta, \chi, \phi, L_0, L_3,  \eta, R_0, \Lambda,  T, d)$  such that for every $ x, y$ in the different components of $D$ with $|x-y|\ge 1$ and $t \in (0, T]$ we have  
\begin{align*}
  p_D(t, x, y) \ge\, 
c_5\,   \Psi(t,x) \Psi(t,y){te^{-\gamma_2 (5|x-y|/4)^\beta}}\nu(|x-y|).
\end{align*}
\item{\rm (5)}
Suppose in addition that $\beta=\infty$ and $D$ is bounded and connected. 
Then the claim of Theorem \ref{t:nmain} (2) holds.
\end{description}
\end{thm}

Recall that the Green function $G_D(x, y)$ of $Y$ on $D$ is defined as  
$G_D(x, y)
=\int_0^\infty p_D(t, x, y)dt$.
As an application of Theorem \ref{t:nmain} and \ref{t:main}, we derive the sharp two sided estimate on the Green function $G_D(x, y)$ of $Y$ on 
 bounded $C^{1, 1}$ open sets.
For notational convenience, let
\begin{align}
a(x, y):=\sqrt{\phi(\delta_D(x))}\sqrt{\phi(\delta_D(y))}\label{e:axy}
\end{align}
 and
\begin{align*}
g(x,y):=
\begin{cases}
\displaystyle \frac{\phi(|x-y|)}{|x-y|^{d}}\left(1\wedge \frac{\phi(\delta_D(x))}{\phi(|x-y|)}\right)^{1/2}\left(1\wedge \frac{\phi(\delta_D(y))}{\phi(|x-y|)}\right)^{1/2}\qquad&\mbox{ when } d\ge 2,\\
\displaystyle \frac{a(x, y)}{|x-y|}\wedge 
\left(\frac{a(x, y)}{\phi^{-1}(a(x, y))}
+\left(\int_{|x-y|}^{\phi^{-1}(a(x, y))} \frac{\phi(s)}{s^2}ds\right)^{+}\right)\qquad&\mbox{ when } d=1
\end{cases}
\end{align*}
where  $x^+:=x\vee 0$.

\begin{thm}\label{C:green}
Suppose that $D$ is a bounded $C^{1,1}$ open set in $\R^d$ with characteristics $( R_0, \Lambda)$. 
Let $Y$ be a symmetric pure jump Hunt process whose jumping intensity kernel $J$  satisfies {\bf(K$_\eta$)}  and {\bf(SD)}. 
Suppose either 
(1)
the jumping intensity kernel $J$ satisfies the condition {\bf(J1)}, or  
(2) $D$ is connected and the jumping intensity kernel $J$ satisfies the condition {\bf (J2)} with $\beta=\infty$.
Then for every $(x, y)\in D\times D$, we have $  G_D(x, y)\asymp g(x, y).$
\end{thm}

\begin{remark}
When $d=1$,
 if either $\la<1$ or $\ua>1$, one can write the Green function estimates in simpler forms. (see, \cite[Corollary 7.4 and Remark 7.5]{MR3237737})
\end{remark}

In addition, we obtain the {\it uniform and scale-invariant} boundary Harnack inequality with {\it  explicit decay rates} in $C^{1,1}$ open sets as an application of Theorems \ref{t:main} and \ref{C:green}(1). 
A function $f:\R^d \to \R$ is said to be {\it harmonic} in the open set $D$ with respect to $Y$  if for every open set $U\subset D$ whose closure is a compact subset of $D$, 
{$\E_x[|f|(Y_{\tau_U})]<\infty$ for every  $x\in U$ and 
\begin{equation}\label{eq:harmDef}f(x)=\E_x[f(Y_{\tau_U})]\qquad \mbox{ for every } x\in U.\end{equation}
It is  said  that $f$ is regular harmonic in $D$ with respect to $Y$ if   $f$ is harmonic in $D$ with respect to $Y$ and \eqref{eq:harmDef} holds for $U=D$.}

\medskip

The next condition guarantees that $C_c^2(\R^d)$ is in the domain of the Feller generator. 
\begin{itemize}
\item[{\bf (L)}]
$Y$ is Feller and there exists a function $q(r)$ such that $J(x, y)\le q(|x-y|)$ and
\begin{align}
\label{e:Rhqdh}
\lim_{R\to \infty}\int_{|h|>R} q(|h|)dh  =0.
\end{align}
\end{itemize}
Note that if {\bf(J2)} holds, then clearly \eqref{e:Rhqdh} holds, and $Y$ is  Feller based on Theorem \ref{T:2.1}. Thus, condition {\bf (L)} is weaker than  condition {\bf(J2)}.

The next condition on $J$ is necessary for the boundary Harnack inequality to hold (see \cite[Assumption C and Example 5.14]{MR3271268}).
\begin{itemize}
\item[{\bf (C)}]
For any  $0<r<R \le 2$ there exists $C^*=C^*(\phi, d, r/R)$ such that for any  $x_0 \in \Rd$,  $x\in B(x_0, r)$ and $y\in B(x_0, R)^c$,
$(C^*)^{-1}J(x_0, y)\le J(x, y)\le C^* J(x_0, y).$
\end{itemize}
Note that when $\beta\in [0,1]$, condition {\bf(J2)} implies {\bf (C)}. 
In contrast, when $\beta\in (1, \infty]$, the boundary Harnack inequality does not hold under condition {\bf(J2)}.

\begin{thm}\label{C:BHI}
Suppose $D$ is a $C^{1,1}$ open set in $\R^d$ with characteristics $( R_0, \Lambda)$. 
Let  $Y$ be a symmetric pure jump Hunt process whose jumping intensity kernel $J$ satisfies conditions {\bf(J1)}, {\bf (L)}, {\bf (C)}, {\bf(K$_\eta$)}, and {\bf(SD)}.
Then,  there exists $c=c( \phi,  L_0, L_3, \eta, \Lambda, d)$ such that
for any $0<r<R_0\wedge 1$, $z\in \partial D$ and any non-negative function $f$ in $\R^d$ that is {regular} harmonic in $D\cap B(z, r)$ with respect to $Y$, and vanishes in $D^c\cap B(z, r)$, we have 
\begin{align*}
\frac{f(x)}{f(y)}\le c \sqrt{\frac{\phi(\delta_D(x))}{\phi(\delta_D(y))}}\qquad \mbox{ for any}\,\,\, x, y\in D\cap B(z, r/2).
\end{align*}
\end{thm}

The rest of this paper is organized as follows.
In Section \ref{genY}, we solve a martingale-type problem for $Y$  which yields a Dynkin-type formula. 
Section \ref{prZ} deals with the isotropic L\'evy process $Z$ with L\'evy measure $\nu(|x|)dx$.  We  compute some key upper bounds of  the generator of $Z$ on our testing function for $C^{1, 1}$ open sets.
In Section \ref{sec:exitY}, we present the key estimates of exit distributions (Theorem \ref{T:exit}). 
Section \ref{sec:upbd} contains the proof of the upper bound of $p_D(t,x,y)$.
We use Meyer's construction when $|x-y|<c$.
Then, by using Lemma~\ref{L:4.1} twice, we prove the upper bound of $p_D(t,x,y)$ without using the lower bound of $p(t,x,y)$.
In Sections \ref{sec:plbd} and \ref{sec:lbd}, we prove the lower bound estimates for $p_D(t,x,y)$. 
First, we consider the case $\delta_D(x)\wedge \delta_D(y)\ge t^{1/\alpha}$; that is, $x$ and $y$ are kept away from the boundary of $D$.
These results are presented in Section \ref{sec:plbd} and the key estimates of  the exit distributions obtained in Section \ref{sec:exitY} are used in Section \ref{sec:lbd} to prove the lower bound for all $x, y \in D$.
Finally, in Section \ref{sec:bhi}, as an application of Theorem \ref{T:exit}, we derive the Green function estimates and 
the uniform scale-invariant Boundary Harnack inequality with explicit decay rates in $C^{1,1}$ open sets.

Throughout the rest of this paper, positive constants $L_0, L_1, L_2, L_3, \gamma_1, \gamma_2$ can be regarded as fixed. 
In the statements and the proofs of results, constants $c_i=c_i(a,b,c,\ldots)$, $i=1,2,3, \dots$, denote generic constants that depend on $a, b, c, \ldots$, the exact values of which are unimportant.
These are given anew in each statement and each proof.
The dependence of the constants on the dimension $d \ge 1$ is not be mentioned explicitly.

For  a function space $\H(U)$
on an open set $U$ in $\R^d$, we let   
$\H_c(U):=\{f\in\H(U): f \mbox{ has  compact support}\},$
$\H_0(U):=\{f\in\H(U): f \mbox{ vanishes at infinity}\}$ and $\H_b(U):=\{f\in\H(U): f \mbox{ is bounded}\}$.

\section{Generator of $Y$}\label{genY}

In this section, we assume that $Y$ is the symmetric pure jump Hunt process with the jumping intensity kernel $J$ satisfying the conditions {\bf(J1.1)}, {\bf(J1.2)} and \textbf{(K$_{\eta}$)}. Recall that these conditions imply that $Y$ is strong Feller (see \cite[Theorem 3.1]{MR2524930}).

We define an operator $\gener$ by
\begin{align}
\label{e:opgen}
\gener g(x):=P.V.\int(g(x+h)-g(x))J(x,x+h)dh:=\lim_{\varepsilon\,\downarrow 0}\gener^{\varepsilon}g(x),
\end{align}
where $$\gener^{\varepsilon}g(x):=\int_{|h|>\varepsilon}(g(x+h)-g(x))J(x,x+h)dh,$$
whenever these exist pointwise. 
Let $g\in C_c^2(\Rd)$ and
$\varepsilon<r<1$,  then by {\bf(J1.1)}, we have
\begin{align*}
\gener^\varepsilon &g(x)=\kappa(x,x)\int_{\varepsilon<|h|<r}(g(x+h)-g(x)-h\cdot\nabla g(x))\nu(|h|)dh\\
&+\int_{\varepsilon<|h|<r}(g(x+h)-g(x))(\kappa(x,x+h)-\kappa(x,x))\nu(|h|)dh\\
&+\int_{r\le|h|\le 1}(g(x+h)-g(x))\kappa(x, x+h)\nu(|h|)dh+\int_{1<|h|}(g(x+h)-g(x))J(x, x+h)dh.
\end{align*}
Since \textbf{(K$_{\eta}$)} holds, we have that 
$$|(g(x+h)-g(x))(\kappa(x,x+h)-\kappa(x,x))|\leq ||\nabla g||_{\infty}(L_3+2L_0)|h|^{\eta+1}, \quad x,h\in \Rd.$$
Since {\bf(WS)} and the  inequality $\eta> \ua/2> \ua-1$ imply $\int_{|h|<r}|h|^{2} \nu(|h|)dh\le \int_{|h|<r}|h|^{\eta+1} \nu(|h|)dh<\infty$, $\gener g$ is well defined and $\gener^\varepsilon g$ converges to $\gener g$ locally {uniformly} on $\Rd$.
 Furthermore, for every $0<r<1$,
\begin{align}\label{g:kappa_eta}
\gener &g(x)=\kappa(x,x)\int_{|h|<r}(g(x+h)-g(x)-h\cdot\nabla g(x))\nu(|h|)dh\nn\\
+&\int_{|h|<r}(g(x+h)-g(x))(\kappa(x,x+h)-\kappa(x,x))\nu(|h|)dh\nn\\
+&\int_{r\le|h|\le 1}(g(x+h)-g(x))\kappa(x, x+h)\nu(|h|)dh+\int_{1<|h|}(g(x+h)-g(x))J(x, x+h)dh.
\end{align}

Now we are ready to prove the following lemma. 

\begin{lem}\label{genBound}
There is $C_{\ref{genBound}}=C_{\ref{genBound}}(\phi,  \eta, L_0,L_3)>0$ such that for any
function  $g\in C_c^2(\Rd)$ and $0<r<1$,
\begin{align}\label{e:genB}
||\gener g||_\infty\leq \frac{C_{\ref{genBound}}}{\phi(r)}\left( r^2||\partial^2 g||_\infty+r^{\eta+1}||\nabla g||_\infty+||g||_\infty\right).
\end{align}
\end{lem}

\begin{proof}
By \eqref{g:kappa_eta}, \eqref{a:kappa}, \textbf{(K$_{\eta}$)} and {\bf(J1.2)}, we obtain that
\begin{align}\label{kappa_eta}
|\gener g(x)|\le  c_0\Big(& L_0 ||\partial^2 g||_\infty \int_0^r \frac{s}{\phi(s)} ds  +L_3||\nabla g||_\infty\int_0^r \frac{s^{\eta}}{\phi(s)} ds\nn\\
&+ 2 L_0 ||g||_\infty\int_r^1 \frac{ds}{s\phi(s)} \Big) + c_1||g||_{\infty}.
\end{align}
For $s\le r$, since $\phi(r)/\phi(s)\le \uC (r/s)^{\ua}$ by {\bf(WS)} and $\eta>  \ua/2> \ua-1$, we have
\begin{align}\label{kappa_s2}
\int_0^r \frac{s^{\eta}}{\phi(s)} ds\le \frac{\uC}{\phi(r) } \frac{1}{\eta+1-\ua}r^{\eta+1}.
\end{align}
For $r<s$,  since $ \lC(s/r)^{\la}\le \phi(s)/\phi(r)$ by  {\bf(WS)}, we have 
\begin{align}\label{kappa_l}
\int_r^{\infty} \frac{ds}{s\phi(s)}\le \frac{\lC^{-1}}{\phi(r)}r^{\la}\int_r^{\infty} s^{-1-\ua}ds<(\lC\, \la)^{-1}\frac{1}{\phi(r)}.
\end{align}
Applying \eqref{kappa_s2} and \eqref{kappa_l} to  \eqref{kappa_eta}, we conclude that \eqref{e:genB} hold.
\end{proof}

\begin{lem}\label{Dynkin_formula}
For any $u\in C^2_c(\Rd)$ and $x\in \R^d$, there exists a $\p_x$-martingale $M^u_t$ with respect to the filtration of $Y$ such that 
$$M^u_t=u(Y_t)-u(Y_0)-\int_0^{t}  {\gener} u(Y_s) ds$$
 $\p_x$-a.s. In particular, for any 
stopping time $S$ with $\E_x S<\infty$
 we have
\begin{align}
\label{e:SnS}
\E_x u(Y_{S})-u(x)=\E_x\int^{S}_0\gener u(Y_s)ds.
\end{align}
\end{lem}

\begin{proof}
Let $(A, D(A))$ be the $L^2$-generator of the semigroup $T_t$ with respect to $Y$.
Due to \cite[Proposition 2.5]{MR2365348}, we have $C_c^2(\Rd) \subset D(A)$ and $A|_{C_c^2(\Rd)}=\gener |_{C_c^2(\Rd)}$.
Since $T_t$ is strongly continuous (see \cite[Theorem 1.3.1 and Lemma 1.3.2]{MR569058}) we have that  for any $f\in D(A)$ and $t\geq 0$,
$$\left|\left|(T_tf-f)-\int^t_0T_sAfds\right|\right|_{L^2}=0$$
(see e.g. \cite[Proposition 1.5]{MR838085}). 
Hence for $u\in C_c^2(\Rd)$, 
\begin{equation}\label{DF1}
T_tu(x)-u(x)=\int^t_0T_s \gener u(x)ds,\quad a.e.\,\, x\in \Rd,
\end{equation}
and $\gener u$ is bounded by Lemma \ref{genBound}.

Let us denote $g_t(x)=\int^t_0T_s\gener u(x)ds$. First, we show that $g_t\in C_b(\Rd)$, $t>0$. Note that  $g_t(x)\leq t||\gener u||_\infty$. 
Hence, since $T_\varepsilon$ is strong Feller for any $\varepsilon>0$, 
we have $T_\varepsilon g_{t-\varepsilon}\in C_b(\Rd)$ for all $\varepsilon\in (0, t)$. 
Moreover,
$$|g_t(x)-T_\varepsilon g_{t-\varepsilon}(x)|=|g_\varepsilon(x)|\leq \varepsilon ||\gener u||_\infty.$$
Hence, $g_t$ is continuous and  \eqref{DF1} holds for any $x\in\Rd$. 
This and Markov property imply that
$$M^u_t=u(Y_t)-u(Y_0)-\int^t_0\gener u(Y_s)ds$$ 
is $\p_x$-martingale for any $x\in\Rd$. 
Since $|M^u_t|\leq 2||u||_\infty+t||\gener u||_\infty$, by the optional stopping theorem \eqref{e:SnS} follows.\end{proof}

\begin{lem}\label{st5n}
There exists a constant $C_{\ref{st5n}}=C_{\ref{st5n}}(\phi,  \eta,  L_0, L_3 )>0$ such that, for any $r \in (0,1]$, $x_0\in\Rd$, and any stopping time $S$ (with respect to the filtration of $Y$), we have
$$
\mathbb{P}_x\left(|Y_{S}-x_0|\ge r\right) \,\le\, C_{\ref{st5n}} \,\frac{ \E_x[S]}{\phi(r)},
\qquad x \in  B(x_0, r/2).
$$
\end{lem}

\begin{proof} 
Fix $x_0\in \Rd$. 
Since this lemma is clear for $\E_x[S]=\infty$, we consider the case that $\E_x[S]<\infty$ for $x \in B(x_0, r/2)$.
Define a radial function $g\in C^{\infty}_c(\R^d)$ such that $-1\le g\le 0$, with 
\begin{align*}
g(y):=\left\{\begin{array}{lll}-1, & \text{if } |y|< 1/2\\
0, & \text{if }  |y|\ge 1.\\
\end{array}
\right.
\end{align*}
Then, 
\begin{align*}
\sum^d_{i=1}\left\|\frac{\partial}{\partial y_i}
g\right\|_\infty+ \sum^d_{i, j=1}\left\|\frac{\partial^2}{\partial y_i\partial y_j}
g\right\|_\infty 
=c_1
<\infty .
\end{align*}
For any $r \in (0,1]$, define $g_{ r}(y)=g(\frac{x_0-y}{r})$ so that $-1\le g_{ r}\le 0$,
\begin{equation}\label{e:n1}
g_{ r}(y)=
\begin{cases}
-1, & \text{if }\,|x_0-y|< r/2\\
0, & \text{if }\, |x_0-y|\ge r, 
\end{cases}
\end{equation}
and 
\begin{align}\label{e:n2}
\sum^d_{i=1}\left\|\frac{\partial}{\partial y_i} g_{ r}\right\|_\infty \,<\, c_1\, r^{-1} \quad \text{and} \quad  \sum^d_{i, j=1}\left\|\frac{\partial^2}{\partial y_i\partial y_j} g_{r}\right\|_\infty \,<\, c_1\, r^{-2}.
\end{align}
By Lemma \ref{genBound}, there exists $c_2=c_2(\la, \lC, \ua,\uC, \eta,L_0,L_3)>0$ such that for $0< r<1$,
\begin{align}\label{e:genb}
||\gener g_{r}||_\infty\le \frac{ c_2}{ \phi(r)}.
\end{align}

Combining Lemma \ref{Dynkin_formula}, \eqref{e:n1} and \eqref{e:genb}, we find that for any $x\in B(x_0, r/2)$ with $\E_xS<\infty$,  we have
\begin{align*}
&\p_x\(|Y_{S}-x_0|\ge r\)
=\E_x\left[1+g_{ r} \left(Y_{S}\right);|Y_{S}-x_0|\ge r\right]\\
&\le \E_x\left[1+g_{ r} \left(Y_{S}\right)\right]=-g_r(x)+\E_x\left[g_{ r} \left(Y_{S}\right)\right]= \E_x\left[ \int_0^{S}   {\gener} g_{ r}(Y_t)dt \right] \le \|{\gener}g_{ r}\|_\infty\,  \E_x[S]\\&\le c_2\frac{\E_x[S]}{\phi(r)}.
\end{align*}
\end{proof}

Recall that for any open set $D\subset \R^d$, $\tau_D= \inf\{t>0: Y_t\notin D\}$ denote the first exit time of $D$ by the process $Y_t$.

\begin{cor}\label{st5}
There exists a constant $C_{\ref{st5}}=C_{\ref{st5}}(\phi,  \eta,  L_0, L_3 )>0$ such that, for any $r \in (0,1]$, $x_0\in\Rd$, and any open sets $U$ and $D$ with $D \cap B(x_0, r ) \subset U  \subset  D$, we have
\begin{align}\label{e:st5}
\mathbb{P}_x\left(Y_{\tau_U} \in D\right) \,\le\, C_{\ref{st5}} \,\frac{ \E_x[\tau_U]}{\phi(r)},
\qquad x \in D\cap B(x_0, r/2).
\end{align}
\end{cor}
\begin{proof} 
Since $D\setminus U\subset B(x_0, r)^c$, by Lemma \ref{st5n} we have that for $x \in D\cap B(x_0, r/2)$
\begin{align*}
\p_x\left(Y_{\tau_U} \in D\right)\,\le\,    \p_x\(|Y_{\tau_U}-x_0|\ge r\)      \,\le\,    C_{\ref{st5n}} \frac{ \E_x[\tau_U]}{\phi(r)}.
\end{align*}
\end{proof}

\section{Analysis on $Z$}\label{prZ}

Recall that $Z$ is a pure jump isotropic unimodal L\'{e}vy process with L\'{e}vy measure $\nu(|x|)dx$. Moreover, we assume {\bf(SD)} holds in this section.
The L\'evy-Khintchine (characteristic) exponent of $Z$ has the form
\begin{equation}\label{characFun}
\psi(|\xi|)=\int_\Rd \left(1- \cos (\xi \cdot x) \right) \nu(|x|)dx, \quad \xi\in\Rd.
\end{equation}

Let $Z^d$ be the last coordinate of $Z$ and  $M_t=\sup_{s\le t}Z^d_s$ and let $L_t$ be the local time {at $0$} for
{$M_t-Z^d_t$}, the last coordinate of $Z$ reflected at the supremum.
We consider its right-continuous inverse, $L^{-1}_s$ which is called the ascending ladder time process for $X^1_t$. 
Define the ascending ladder-height process as
$H_s = Z_{L^{-1}_s}^{d} = M_{L^{-1}_s}$.
The Laplace exponent of $H_s$ is
\begin{equation*}\label{exL}
\kappa(\xi)= \exp\left\{\frac{1}{\pi} \int_0^\infty \frac{ \log {\psi}(\theta \xi)}{1 + \theta^2} \, d\theta\right\}, \quad \xi\ge 0.
\end{equation*}
(See \cite[Corollary 9.7]{MR0400406}.)
The renewal function $V$ of the ascending ladder-height process $H$ is defined as
\begin{equation*}\label{e:defV}
V(x) = \int_0^{\infty}\p(H_s \le x)ds, \quad
x {\in \R}.
\end{equation*}
then $V(x)=0$ if $x<0$ and $V$ is non-decreasing.  Also $V$ is subadditive (see \cite[p.74]{MR1406564}), that is,
\begin{equation}\label{subad}
 V(x+y)\le V(x)+V(y), \quad x,y \in \R,
\end{equation}
 and $V(\infty)=\infty$.
Since {the distribution of $Z^d_t$ is absolutely continuous for every $t>0$ the resolvent measures of $Z^d_t$ as well}, (see \cite[Theorem 6]{MR0341626}), 
it follows by \cite[Theorem 2]{MR573292}
that $V(x)$ is absolutely continuous and harmonic on $(0,\infty)$ for the process $Z_t^d$.
Also, $V^\prime$ is a positive harmonic function for $Z_t^d$ on $(0,\infty)$, hence $V$ is actually (strictly) increasing.

For $r>0$, define Pruitt's function $h(r)=\int_{\R^d} \(1 \wedge {|z|^2}r^{-2}\)\nu(dz)$ (e.g., see  \cite {MR632968}). By \cite[Corollary 3]{2013arXiv1305.0976B} and \cite[Proposition 2.4]{BGR1} (see also \cite[p.74]{MR1406564}), we first note that
\begin{equation*}\label{compVhp}
h(r)\asymp [V(r)]^{-2}\asymp \kappa(r^{-1})\asymp \psi(r^{-1})\qquad\mbox{ for any }\,\, r>0.
\end{equation*}
Clearly,  {\bf(WS)} implies that 
$s\to \phi(s^{-1})^{-1}$ also satisfies  {\bf(WS)}, that is, using the notation in \cite{2013arXiv1305.0976B}, 
$s\to \phi(s^{-1})^{-1}\in\WLSC{\la}{\lC}{0}\,\cap \,\WUSC{\ua}{\uC}{0}$.
So by \eqref{d:nu} and \cite[Proposition 28]{2013arXiv1305.0976B}, we have that 
$$\psi(r)\asymp \phi(r^{-1})^{-1}\qquad \mbox{ for any }\,\,  r>0.$$
Combining these observations, we conclude that 
\begin{align}\label{e:com}
V(r)\asymp [\phi(r)]^{1/2} \quad \mbox{ and} \quad 
\nu(r) \asymp [V(r)]^{-2} r^{-d}
\qquad \mbox{ for any } \,\, r>0.
\end{align}
So by {\bf(WS)},  there exists $C_V:=(\lC, \uC, d)>1$ such that 
\begin{align}\label{V:sc}
C_V^{-1} \left(\frac{R}{r}\right)^{\la/2}\le \frac{V(R)}{V(r)}\le C_V \left(\frac{R}{r}\right)^{\ua/2}\qquad \mbox{ for any } 0<r\le R.
\end{align}

Define $w(x):=V((x_d)^+)$ and $\H:=\{x=(\widetilde{x},x_d) \in \R^d:x_d>0 \}$. Since the renewal function $V$ is harmonic on $(0, \infty)$ for $Z^d$, by the strong Markov Property $w$ is harmonic in $\H$ with respect to $Z$.
\begin{prop}\label{VbisEst} 
$x \to V(x)$ is twice-differentiable  for any $x>0$, and there exists $C_{\ref{VbisEst}}>0$ such that $$|V''(x)|\leq C_{\ref{VbisEst}}\frac{V'(x)}{x\wedge 1}\quad \text{and}\quad V'(x)\leq
C_{\ref{VbisEst}}\frac{V(x)}{x\wedge 1}, \quad x>0.$$
\end{prop}
\begin{proof}
Let $f((\wt y,x_d))=V'((x_d)^+)$ for $\wt y\in \R^{d-1}$. Then $f$ is harmonic in $\H$. 
The assumption \cite[{\bf (A)}]{TKMR} is satisfied by {\bf(SD)}. Hence, by Theorem 1.1 therein, we get for any $x>0$
$$\left|\frac{\partial}{\partial x_d}f((\wt 0,x))\right|\leq C_{\ref{VbisEst}}\frac{f((\wt 0,x))}{x\wedge 1}\quad \text{and}\quad \frac{\partial}{\partial x_d}w(\wt 0,x)\leq
C_{\ref{VbisEst}}\frac{w(\wt 0,x)}{x\wedge 1}.$$
These imply the claim of proposition, because $V(x)=w(\wt 0,x)$ and $V'(x)=f(\wt 0,x)$, $x>0$.

\end{proof}

\begin{prop}\label{st2}
For $\lambda>0$, there exists $C_{\ref{st2}}=C_{\ref{st2}}(d,\lambda)>0$ such that for any $r>0$, we have
\begin{align}
\sup_{\{x\in \R^d\, :\, 0<x_d\le \lambda r\}}\int_{B(x, r)^c} w(y)\nu(|x-y|)dy<  \frac{C_{\ref{st2}}}{V(r)}.
\end{align}
\end{prop}

\begin{proof}
Since $w(x+z)=V(x_d+z_d)\le V(x_d)+V(|z|)$ for $x_d>0$, it follows that
\begin{align*}
\int_{B(x, r)^c} w(y)\nu(|x-y|)dy&=\int_{B(0, r)^c} w(x+z)\nu(|z|)dz\\
&\leq V(x_d)\int_{B(0, r)^c}\nu(|z|)dz+\int_{B(0, r)^c} V(|z|)\nu(|z|)dz.
\end{align*}
By \cite[(2.23) and Lemma 3.5]{BGR1}, we have that 
$$\sup_{\{x\in \R^d\, :\, 0<x_d\le \lambda r\}}\int_{B(x, r)^c} w(y)\nu(|x-y|)dy\leq c_1 \(\frac{V(\lambda r)}{[V(r)]^2}+\frac{1}{V(r)}\).$$
Since $V$ is subadditive,  $V(\lambda r)\le (\lambda +1)V(r)$  for any $\lambda>0$, and therefore we conclude the result.
\end{proof}

For any function $f:\R^d\to \R$ and $x\in \R^d$, we define an operator as follows:
\begin{align*}
\gener_Z f(x)&:=P.V.\int_{\Rd} (f(y)-f(x))\nu(|x-y|)dy:=\lim_{\varepsilon\downarrow 0}\gener_Z^{\eps} f(x),\\
{\cal D} (\gener_Z)&:=\left\{f\in C^{2}(\R^d):P.V.\int_{\Rd} (f(y)-f(x))\nu(|x-y|)dy\,\mbox{ exists and is finite.}\right\},
\end{align*}
where
\begin{align}
\label{e:gZ}
\gener_Z^{\eps} f(x):=\int_{B(x, \varepsilon)^c} (f(y)-f(x))\nu(|x-y|)dy.
\end{align}

Recall that $C_0^2(\Rd)$ be the collection of $C^2$ functions in $\R^d$ vanishing at infinity. It is well known that $C_0^2(\R^d)\subset {\cal D} (\gener_Z)$ and that, by the rotational symmetry of $Z$,
$A_Z|_{C_0^2(\Rd)}=\gener_Z|_{C_0^2(\Rd)}$ where  $A_Z$ is the infinitesimal generator of $Z$(e.g. \cite[Theorem 31.5]{MR1739520}).
Hence, we see that Dynkin formula holds for $\gener_Z$: for each $g \in C_0^{2}(\R^d)$ and any bounded open subset $U$ of $\R^d$ we have
\begin{equation} \label{n_Dynkin_formula}
\E_x\int_0^{\tau_U}  {\gener_Z } g(Z_t) dt =\E_x[g(Z_{\tau_U})]- g(x).
\end{equation}

\begin{thm}\label{st3}
For any $x\in \H$, $\gener_{Z} w(x)$ is well-defined and $\gener_Z w(x)=0$.
\end{thm}

\begin{proof}
By subadditivity of $V$, $|w(y)-w(x)|\leq V(|y_d-x_d|)\leq V(|x-y|)$ for $x\in \H$.
By \cite[Lemma 3.5]{BGR1}, it follows that for any  $\varepsilon\in (0, 1/2)$ 
\begin{align}\label{supgen}
&\left|\int_{B(x, \varepsilon)^c}(w(y)-w(x))\nu(|x-y|)dy\right|\le \int_{B(0, \varepsilon)^c} V(|z|)\nu(|z|)dz<\frac{c_1 }{V(\varepsilon)}<\infty.
\end{align}
Thus $\gener_Z^{\eps} w(x)$ is well defined in $\H$ and
$$\gener_Z^\varepsilon w(x)=\int_{B(x, \varepsilon)^c}(w(y)-w(x)-{\bf 1}_{\{|x-y|<1\}}(x-y)\cdot\nabla w(x))\nu(|x-y|)dy.$$
Since Proposition \ref{VbisEst} implies $V^{''}(s)$ exists and so $w$ is twice differentiable in $\H$,  we have that 
$$x\mapsto \int_{B(x, \varepsilon)}(w(y)-w(x)-(x-y)\cdot\nabla w(x))\nu(|x-y|)dy$$
converges to $0$ locally uniformly in $\H$ as $\varepsilon \downarrow 0$. 
From \eqref{supgen}, we see that  $\gener_{Z}^\varepsilon w(x)$ converges to 
$$\gener_Z w(x)=\int_{\Rd}(w(y)-w(x)-{\bf 1}_{\{|x-y|<1\}}(x-y)\cdot\nabla w(x))\nu(|x-y|)dy$$
locally uniformly in $\H$ as $\varepsilon \downarrow 0 $.

For every $x\in \H$,  $z\in B(x, (\varepsilon\wedge x_d)/2)$ and $y\in B(z, \varepsilon)^c$, it holds that $|y-z|/2\le |x-y|\le 3|y-z|/2$.
Since $r\to \nu(r)$ is decreasing, using Proposition \ref{VbisEst} 
\begin{align*}
&{\bf 1}_{\{|y-z|>\varepsilon\}}\left|w(y)-w(z)-{\bf 1}_{\{|y-z|<1\}}(y-z)\cdot\nabla w(z)\right|\nu(|y-z|)\nn\\
&\leq  c_2\left(\sup_{\eps/2<s <x_d+2} V^{\prime\prime}(s)\right)|x-y|^2 {\bf 1}_{\{\varepsilon/2<|x-y|<2\}}\nu(|x-y|/2)\nn\\
&\qquad +(w(y)+V(x_d+1)) {\bf 1}_{\{|x-y|>1/2\}}\nu(|x-y|/2).
\end{align*}
So applying the dominated convergence theorem with Proposition \ref{st2} and the fact that $\nu$ is a L\'evy density, we obtain that  $x\to \gener_Z^{\eps}w(x)$ is continuous for each $\eps$. 
Therefore, the function $\gener_Z w(x)$ is continuous in $\H$.

Let $U_1$ and $U_2$ be relatively compact open subsets of $\H$ satisfying $\ol{U_1}\subset U_2\subset \ol{U_2}\subset \H$ and $0<r_0:=\dist (U_1, U_2^c)<1 $. By Proposition \ref{st2}, 
\begin{align}\label{h1}
\int_{U_1}\int_{U_2^c} w(y)\nu(|x-y|)dydx \le& \,\,|U_1| \sup_{x\in U_{1}}\int_{U_2^c} w(y) \nu(|x-y|)dy\nn\\
\le& \,\,|U_1| \sup_{x\in U_1}\int_{B(x, r_0)^c} w(y) \nu(|x-y|)dy<\infty.
\end{align}
Since $w$ is harmonic, $w(Z_{\tau_{U_1}})\in L^1(\p_x)$ and 
\begin{align}\label{h2}
\sup_{x\in U_1} \E_x[{\bf 1}_{U_2^c}(Z_{\tau_{U_1}})w(Z_{\tau_{U_1}})]\le \sup_{x\in U_1} \E_x[w(Z_{\tau_{U_1}})]=\sup_{x\in{U_1}} w(x)<\infty.
\end{align}
From \eqref{h1} and \eqref{h2}, the conditions \cite[(2.4), (2.6)]{MR2515419} hold and by \cite[Lemma 2.3, Theorem 2.11(ii)]{MR2515419}, we have that for any $f\in C_c^2(\H)$,
\begin{align}\label{h3}
0=\int_{\Rd}\int_{\Rd} (w(y)-w(x))(f(y)-f(x)) \nu(|x-y|) dxdy.
\end{align}
For $f\in C_c^2(\H)$ with $\supp(f)\subset\ol{U_1}\subset U_2\subset \ol{U_2}\subset \H$,
\begin{align}\label{h4}
&\int_{\Rd}\int_{\Rd} |w(y)-w(x)||f(y)-f(x)| \nu(|x-y|) dxdy\nn\\
=&\int_{U_1}\int_{U_2} |w(y)-w(x)||f(y)-f(x)| \nu(|x-y|) dxdy+2\int_{U_1}\int_{U_2^c} |w(y)-w(x)||f(y)| \nu(|x-y|) dxdy\nn\\
\le& c_2\int_{U_1\times U_2}|x-y|^2 \nu(|x-y|) dxdy+2||f||_{\infty}\int_{U_1}\int_{U_2^c} w(y) \nu(|x-y|) dxdy\nn\\
&\qquad +2||f||_{\infty}|U_1|\left(\sup_{y\in U_1}w(y)\right)\int_{U_1}\int_{U_2^c} \nu(|x-y|) dxdy,
\end{align}
and it is finite from \eqref{h1} and the fact that $\nu$ is a L\'evy density. 
Applying the dominated convergence theorem with \eqref{h3} and  \eqref{h4}, for any $f\in C_c^2(\H)$, we have
\begin{align*}
0=&\lim_{\eps\downarrow 0}\int_{\{(x,y)\in\Rd\times\Rd: |x-y|>\eps\}} (w(y)-w(x))(f(y)-f(x)) \nu(|x-y|) dxdy\\
=&-2 \lim_{\eps\downarrow 0}\int_{\H} f(x)\left(\int_{\{y\in\Rd: |x-y|>\eps\}} (w(y)-w(x)) \nu(|x-y|)dy\right)dx=-2\int_{\H}f(x)\gener_Zw(x)dx.
\end{align*}
We have used Fubini's theorem and the fact that $\gener_Z^{\eps}w\to \gener_Z w$ converges uniformly on the support $f$. Hence, by the continuity of $\gener_Z w$, we have $\gener_Z w(x)=0$ in $\H$.
\end{proof}

\begin{prop}\label{g:st1}
Suppose that $D$ is a $C^{1,1}$ open set in $\R^d$ with characteristics $(R_0, \Lambda)$.
For any $z\in \partial D$ and $r\le 1\wedge R_0$, we define 
$$h_r(y)=h_{r,z}(y):=V(\delta_D(y)){\bf 1}_{D\cap B(z, r)} (y).$$
There exists $C_{\ref{g:st1}}=C_{\ref{g:st1}}(\phi,  \Lambda, d)>0$ independent of $z$ such that $\gener_{Z} h$ is well-defined in $D\cap B(z, r/4)$ and 
\begin{align}\label{gener}
|\gener_{Z} h_r(x)|\le \frac{C_{\ref{g:st1}}}{V(r)}\,\qquad\qquad \mbox{ for all } x\in D\cap B(z, r/4). 
\end{align}
\end{prop}

\begin{proof}
Since the case of $d=1$ is easier, we give the proof only for $d\ge 2$. 
\kk{ Without loss of generality, we assume that $R_0<1$ and $\Lambda>1$.}
For $x\in D\cap B(z, r/4)$, let $z_x\in \partial D$ be the point satisfying $\delta_D(x)=|x-z_x|$.
Let $\varphi$ be a  $C^{1,1}$ function and $CS=CS_{z_x}$ be an orthonormal coordinate system with $z_x$ chosen so that $\varphi(\wt 0)=0$, $ \nabla\varphi(\wt 0)=(0, \dots, 0)$, $\| \nabla \varphi \|_\infty \leq \Lambda$, $|\nabla \varphi (\wt y)-\nabla \varphi (\wt z)| \leq \Lambda |\wt y-\wt z|$, and $x=(\wt{0}, x_d)$, $ D\cap B(z_x, R_0)=\{y=(\wt{y}, y_d)  \in B(0, R_0) \mbox{ in } CS: y_d>\varphi(\wt{y})\}$.
We fix the function $\varphi$ and the coordinate system $CS$, and we define a function $g_x(y)=V(\delta_{\H}(y))=V(y_d)$, where  $\H=\{y=(\widetilde{y}, y_d) \mbox{  in } CS :y_d>0\}$ is the half space in $CS$. 

Note that  $h_r(x)=g_x(x)$, and  that $\gener_Z (h_r- g_x)=\gener_Z h_r$ by Theorem \ref{st3}.
So, it suffices to show that $\gener_Z (h_r- g_x)$ is well defined and
that there exists a constant $c_1=c_1(\la, \lC, \ua, \uC,  \Lambda,  d)>0$ independent of  $x\in D\cap B(z, r/4)$ and $z \in \partial D$  such that
\begin{align}\label{e:claims}
\int_{D\cup \H}|h_r(y)-g_x(y)|\nu(|x-y|) dy\le c_1 V(r)^{-1}.
\end{align}

We define $\wh \varphi :B(\wt{0}, r)\to \R$ by $\wh \varphi(\wt{y}):=2\Lambda|\wt{y}|^{2}$.
Since $\nabla\varphi (\wt{0})=0$, by the mean value theorem we have 
$-\wh \varphi(\wt{y})\le\varphi(\wt{y})\le \wh \varphi(\wt{y})$  for any $y\in D\cap B(x, r/2)$
and so that
\begin{align*}
\{z=(\widetilde{z}, z_d)\in B(x, r/2) : z_d\ge&\,\, \wh{\varphi}(\widetilde z)\}\subset D \cap  B(x, r/2)\nn\\
&\subset \{z=(\widetilde{z}, z_d)\in B(x, r/2) : z_d\ge -\wh{\varphi}(\widetilde z)\}.
\end{align*}
Let $A:=\{y\in (D\cup \H)\cap B(x, r/4):-\wh{ \varphi}(\widetilde{y})\le y_d\le \wh{\varphi}(\widetilde{y})\}$
and 
$E:=\{y\in B(x, r/4): y_d> \wh{\varphi}(\widetilde{y})\}\subset D$. 
We will prove \eqref{e:claims} by showing that $\mbox{I}+\mbox{II}+\mbox{III}\le c_1 V(r)^{-1}$, where
\begin{align*}
&{\rm I}:=\int_{B(x, r/4)^c}( h_r(y)+g_x(y))\nu(|x-y|)dy,\nn\\
&{\rm II}:=\int_A (h_r(y)+g_x(y))\nu(|x-y|)dy,\qquad \mbox{ and }\qquad {\rm III}:=\int_E |h_r(y)-g_x(y)|\nu(|x-y|)dy.
\end{align*}

First, since $h_r\le V(r)$, by \cite[(2.23)]{BGR1} and Proposition \ref{st2}, we have
\begin{align*}
{\rm I}&\le V(r)
\int_{B(x, {r}/{4})^c} \nu(|x-y|)dy 
+\sup_{\{z\in\R^d:0<z_d<r\}}\int_{B(z, {r}/{4})^c\cap \H} g_x(y)\nu(|z-y|)dy\\
&\le  c_2 V(r)^{-1}+\sup_{\{z\in\R^d:0<z_d<r\}}\int_{B(z, {r}/{4})^c\cap \H} w(y)\nu(|z-y|)dy\le (c_2 +C_{\ref{st2}})V(r)^{-1}.
\end{align*}

For $y\in A$,  since $\wh{\varphi}(\widetilde{y})\le 2\Lambda|\widetilde{y}|$ and $V$ is increasing and subadditive, we observe that  $h_r(y)+g_x(y)\le 2V(2\wh{\varphi} (\widetilde{y}))\le c_3 V(|\widetilde {y}|)$. For $s\le r/4$, note that $m_{d-1}(\{y:|\widetilde{y}|=s, - \wh{\varphi}(\widetilde{y})\le y_d\le \wh{\varphi}(\widetilde{y})\})\le c_4 s^{d}$ where $m_{d-1}(dy)$ is the surface measure on $\R^{d-1}$.
Thus 
$\int_{|\widetilde{y}|=s} {\bf1}_A(y)V(|\widetilde {y}|)\nu( |\widetilde{y}|)$ $m_{d-1}(dy)\le c_4 V(s)\nu(s) s^{d}$ for $0<s<r/4$.
From \eqref{V:sc}, we note that $V(s)^{-1}\le C_V V(r)^{-1} (r/s)^{\ua/2}$ for $s\le r$. 
Hence, by \eqref{e:com} and \eqref{d:nu}, 
 \begin{align*}
{\rm II}\le&\,  c_3 c_4\,\int _0 ^{r}V(s)\nu(s)s^{d} ds\nn\\
\le&  \,c_5\,\int _0 ^{r}V(s)^{-1}s ds\le  c_4\cdot  C_V V(r)^{-1}  r^{\ua/2} \int _0 ^{r}s^{-\ua/2} ds\nn\\
=& \,   c_5\cdot C_V V(r)^{-1} \frac{1}{1-\ua/2} \, r^{1}\le   c_6V(r)^{-1}
\end{align*}
for some positive constant $c_6=c_6(\ua,\lC, \uC,  \Lambda, d)$.

\kk{When $y\in E$, we have that $|y_d-\delta_D(y)|\le \kk{\Lambda|\widetilde{y}|^2}$.
Indeed,  if $0<y_d=\delta_{\H}(y)\le \delta_D(y)$ and $y\in E$,  $ \delta_{D}(y)\le y_d+|\varphi(\widetilde{y})|\le y_d+\kk{\Lambda|\widetilde{y}|^2}$.
Since we assume that $R_0<1$ and $\Lambda>1$, if $y_d=\delta_{\H}(y)\ge \delta_D(y)$ and $y\in E$,
the interior ball condition implies that
\begin{align*}
y_d- \delta_D(y)&\le y_d -R_0+ \sqrt{|\wt y|^2+(R_0-y_d)^2}=\frac{|\wt y|^2}{\sqrt{|\wt y|^2+(R_0-y_d)^2}+(R_0-y_d)}\nn\\
& \le \frac{|\wt y|^2}{2(R_0-y_d)}\le \frac{|\wt y|^2}{R_0}\le \Lambda |\wt y|^2.
\end{align*}
Furthermore, $|y_d-\delta_D(y)|\le \kk{\Lambda|\widetilde{y}|^2}$ yields that 
$$\Lambda |\wt y|^2\le y_d-\Lambda |\wt y|^2\le y_d\wedge \delta_D(y)\quad\mbox{ and } \quad y_d\vee \delta_D(y)-(y_d-\Lambda |\wt y|^2)\le 2\Lambda |\wt y|^2.$$
Hence, by the mean value Theorem  and  the scale invariant Harnack inequality for $Z^d$ (\cite[Theorem 1.4]{MR2524930}) applying to $V^{\prime}$, we have that
\begin{align}\label{e:h1}
|h_r(y)-g_x(y)|=&|V(\delta_D(y))-V(y_d)|\le 
\sup_{u\in [(y_d\wedge \delta_D(y), \,\, y_d\vee \delta_D(y)]} V^{\prime}(u)|\delta_D(y)-y_d|\nn\\
\le &\sup_{u\in [y_d-\Lambda|\wt y|^2, \,\, y_d\vee \delta_D(y)]} V^{\prime}(u)|\delta_D(y)-y_d|\nn\\
\le &\,\,c_7 \inf_{u\in [y_d-\Lambda|\wt y|^2,\,\, y_d\vee \delta_D(y)]} V^{\prime}(u)|\widetilde{y}|^{2}\le 2\Lambda c_7  V^{\prime}\(y_d-\tfrac{1}{2}\wh{\varphi}(\wt{y})\) |\widetilde{y}|^{2}.
\end{align}}
Since  $E\subset \{ (\wt{y}, y_d): |\widetilde{y}|<r/4, \, \wh{\varphi}(\widetilde{y})  <y_d< \kk{\tfrac{1}{2}}\wh{\varphi}(\widetilde{y})+r/2\}$, using with \eqref{e:h1} and the polar coordinates for $|\wt{y}|=v$, we first see that
\begin{align*}
{\rm III} \le&\, 2\Lambda c_7 \int_E V^{\prime}\(y_d-\tfrac{1}{2}\wh{\varphi}(\wt{y})\)|\widetilde{y}|^{2} \nu(|x-y|)dy\nn\\
\le&\, c_8\int_0^{r/4}\int_{\wh{\varphi}(v)}^{\kk{\tfrac{1}{2}}\wh{\varphi}(v)+r/2}V^{\prime}(y_d-\tfrac{1}{2}\wh{\varphi}(v))\nu((v^2+|y_d-x_d|^2)^{1/2}) v^{d}dy_d dv
\end{align*}
Let $s:= y_d-\tfrac{1}{2}\wh{\varphi}(v)$.
Since $(v^2+|y_d-x_d|^2)^{1/2}\ge (v+|y_d-x_d|)/2$ and $\nu$ is decreasing, by \eqref{d:nu} and \eqref{e:com}, we have that
\begin{align*}
 \nu((v^2+|y_d-x_d|^2)^{1/2}) v^{d}&\le \nu((v+|s+\tfrac{1}{2}\wh{\varphi}(r)-x_d|)/2)(v+(|s+\tfrac{1}{2}\wh{\varphi}(v)-x_d|)^{d}\nn\\
 &\le c_9 V(v+|s+\tfrac{1}{2}\wh{\varphi}(v)-x_d|)^{-2}
\end{align*}
For $g(s):=\sup_{u\ge s}({V(u)}/{u})$, $V^{\prime}(s)\le C_{\ref{VbisEst}} ({V(s)}/{s}) \le C_{\ref{VbisEst}} g(s)$ by Proposition \ref{VbisEst}. 
Therefore we have that 
\begin{align*}
{\rm III}&\le c_8\cdot c_9\int_0^{r/2}\int_0^{r/2} V^{\prime}(s)V(v+|s+\tfrac{1}{2}\wh{\varphi}(v)-x_d|)^{-2}dsdv\nn\\
&\le  c_{10}\int_0^{r/2}\int_0^{r/2} g(s)V(v+|s+\tfrac{1}{2}\wh{\varphi}(v)-x_d|)^{-2}dsdv.
\end{align*}
Applying \cite[Lemma 4.4]{2012arXiv1212.3092K} with 
non-increasing functions $g(s)$ and $f(s):= V(s)^{-2}$ and 
$x(r)=s+\tfrac{1}{2}\wh{\varphi}(r)$, we have that  
\begin{align*}
{\rm III}\le c_{11} \int_0^{3r/4}G(u)V(u)^{-2}du.
\end{align*}
where $G(u)=\int_0^u g(s)ds$. By subadditivity of $V$ and \eqref{V:sc},  
$G(u) \le 2\int_0^u ({V(s)}/{s}) \,ds\le c_{12} V(u)$.
Using \eqref{V:sc} again, we conclude that
\begin{align*}
{\rm III} \le c_{11}\cdot c_{12} \int_0^{r} V(u)^{-1}du\le c_{13}V(r)^{-1}r^{\ua/2}\int_0^{r} u^{-\ua/2}du\le c_{14}V(r)^{-1}
\end{align*}
for some positive constant $c_{14}:=c_{14}(\la, \ua, \lC, \uC,  d)$. 
\end{proof}

\section{Estimates on exit distributions for $Y$}\label{sec:exitY}

In this section we give some key estimates on exit distributions for $Y$. 
Throughout this section, we assume that $Y$ is the symmetric pure jump Hunt process with the jumping intensity kernel $J$ satisfying the conditions {\bf(J1.1)}, {\bf(J1.2)}, {\bf (SD)} and \textbf{(K$_{\eta}$)}.

For any $x\in \R^d$, stopping time $S$ (with respect to the filtration of $Y$), and non-negative measurable function $f$ on $\R_+ \times \R^d\times \R^d$ with $f(s, y, y)=0$ for all $y\in\R^d$ and $s\ge 0$, we have the following  L\'evy system:
\begin{equation}\label{e:levy}
\E_x \left[\sum_{s\le S} f(s,Y_{s-}, Y_s) \right] = \E_x \left[ \int_0^S \left(\int_{\R^d} f(s,Y_s, y) J(Y_s,y) dy \right) ds \right]
\end{equation}
(e.g., see  \cite[Appendix A]{MR2357678}).

Throughout this section, we assume that $D$ is a $C^{1,1}$ open set with $C^{1, 1}$ characteristics $(R_0, \Lambda)$, and   
without loss of generality, we will assume that $R_0<1$ and $\Lambda>1$.
Recall that the function $h_r(y)=h_{r,z}(y)$ is defined in Proposition \ref{g:st1}.

\begin{lem}\label{g:st1-1}
Let $r\le R_0/2$.
For any $z\in \partial D$ and $k\in\N$, let $B_k:=\{y \in D \cap B(z, r/4): \delta_{D \cap B(z, r/4) }(y) \ge 2^{-k}\}$. Then, for every $u\in \Rd$ and $k\in\N$ with $|u|<2^{-k}<2^{-8} r$,
\begin{align*}
 \gener^u h_{r,z}(w):= \lim_{\eps\downarrow 0} \, \int_{|(w-u)-y|>\eps} (h_{r,z}(y)-h_{r,z}(w-u))J(w, u+y)dy
\end{align*}
is well defined in $B_k$
and there exists $C_{\ref{g:st1-1}}=C_{\ref{g:st1-1}}(\phi, L_0, L_3, \Lambda,  \eta, d)>0$ independent of $z\in  \partial D$, $k\in\N$  with $2^{-k+8} <r \le R_0$ such that
\begin{align*}
| \gener^u h_{r,z}(w)|\le \frac{C_{\ref{g:st1-1}}}{ V(r)}\qquad \mbox{for all} \,\,\,\, w\in B_k,\,\, |u|<2^{-k}.
\end{align*}
\end{lem}

\proof
We fix $z \in \partial D$ and use the short notation $h_r(y)=h_{r,z}(y)$.  
For any $w\in B_k$ and  $|u|<2^{-k}<2^{-8} r$,  let $x:=w-u\in B(z, r/4)$. 
Define $\kappa_u(x, y):=\kappa(u+x, u+y)$, and for $\eps < 2^{-k-1}$ denote $A_\eps(x)$ and  $L(x)$ by
\begin{align*}
A_\eps(x)&:=\int_{\eps<|x-y|\le 1} (h_r(y)-h_r(x))(\kappa_u(x, y)-\kappa_u(x, x))\nu(|x-y|)dy,\\
 L(x)&:=\int_{1<|x-y|} (h_r(y)-h_r(x)){ (J(x, y)- \kappa_u(x, x) \nu(|x-y|))}dy,
\end{align*} so that 
\begin{align*}
&\int_{\eps<|x-y|\le 1} (h_r(y)-h_r(x))\kappa_u(x, y)\nu(|x-y|)dy+\int_{1<|x-y|} (h_r(y)-h_r(x))J(x, y)dy\\
=&  A_\eps(x)+\kappa_u(x, x)\cdot \gener_Z^{\eps} h_r(x) +L(x)
\end{align*}
where $\gener_Z^{\eps}$ is defined in \eqref{e:gZ}.

By the definition of $h_r$,  \eqref{a:kappa}, \eqref{d:nu} and {\bf(J1.2)}, for $r\le R_0<1$, we first obtain that
\begin{align}\label{e:C(x)}
|L(x)|\le c_0{\left(\left|\int_{1<|x-y|} J(x, y)dy\right|+\left|\int_{1<|x-y|} \nu(|x- y|)dy\right|\right)} \le c_1 V(r)^{-1}.
\end{align}

On the other hand, 
\begin{align*}
{|A_{\eps}(x)|}&\le  \,  \left(\int_{ |x-y|<r/2}+\int_{r/2\le |x-y|}\right)|h_r(y)-h_r(x)||\kappa_u(x, y)-\kappa_u(x, x)|\nu(|x-y|)dy\\
&=: \mbox{I}(x)+\mbox{II}(x).
\end{align*} 
For $|x-y|<r/2$,   $\left|h_r(y)-h_r(x)\right|\le V(|x-y|)$ by subadditivity of $V$, and  $V(r)/V(|x-y|)\le C_V (r/|x-y|)^{\ua/2}$ by \eqref{V:sc}. 
Also $|\kappa_u(x, y)-\kappa_u(x, x)|\le L_3 |x-y|^{\eta}$ by the assumption \textbf{(K$_\eta$)}.
Hence \eqref{d:nu} and \eqref{e:com} imply that 
\begin{align}\label{eq:I}
|\mbox{I}(x)| &\le c_2 \int_{|x-y|<r}V(|x-y|)^{-1}|x-y|^{\eta-d}dy\nn\\
&\le c_2 C_V V(r)^{-1}r^{\ua/2} \int_{|x-y|<r}|x-y|^{-\ua/2+\eta-d}dy\le c_3 V(r)^{-1}
\end{align}
for some positive constant $c_3:=c_3( \ua, \lC, \uC, L_3, \eta, d)$. The last inequality holds since $\eta>\ua/2$.
{To obtain the upper bound of II$(x)$, note that $\left|h_r(y)-h_r(x)\right|\le 2V(|x-y|)$ for $r/2\le |x-y|$.
Indeed, if $y\in (D\cap B(z, r))^c$,  then $\left|h_r(y)-h_r(x)\right|=\left|h_r(x)\right|\le V(r)\le 2V(|x-y|)$ by subadditivity of $V$. If  $y\in  D\cap B(z, r)$,  $\left|h_r(y)-h_r(x)\right|\le V(|\delta_D(y)-\delta_D(x)|)\le V(|x-y|)$ by subadditivity of $V$.}
Hence by \eqref{a:kappa} and \cite[Lemma 3.5]{BGR1}, we obtain that 
{\begin{align}\label{eq:II}
|\mbox{II}(x)|\le 4L_0 \int_{r/2\le |x-y|}V(|x-y|)\nu(|x-y|)dy\le c_5 V(r)^{-1}.
\end{align}}
From  \eqref{eq:I} and \eqref{eq:II}, 
\begin{align}\label{eq:III}
\lim_{\eps\downarrow 0} A_\eps(x) \text{ exists } \quad \text{ and } \quad  |\lim_{\eps\downarrow 0} A_\eps(x)| \le (c_3+c_5) V(r)^{-1}.
\end{align}

Finally  from Proposition \ref{g:st1}, 
$\lim_{\eps\downarrow 0}\gener_Z^{\eps} h_r(x) $ exists and 
\begin{align}\label{eq:A}
\left|\lim_{\eps\downarrow 0}\gener_Z^{\eps} h_r(x)  \right|\le C_{\ref{g:st1}}\,V(r)^{-1}.
\end{align}
Hence combining \eqref{e:C(x)}, \eqref{eq:III} and \eqref{eq:A}, we have the conclusion.
\qed

\begin{thm}\label{T:exit}
For any $x\in D$, let $z_x\in \partial D$ be a point satisfying $\delta_D(x)=|x-z_x|$.
\begin{description}
\item{(1)}
There are constants $A_{\ref{T:exit}}=A_{\ref{T:exit}}(\phi, L_0, L_3, \Lambda,  \eta)\in(0, 1)$ and $C_{\ref{T:exit}.1}=C_{\ref{T:exit}.1}$ $(\phi, L_0, L_3,$ $ \Lambda, \eta)>0$, such that for any $s \le A_{\ref{T:exit}}R_0/2 $ and $x\in D$ with $\delta_D(x)<s$,
\begin{equation}\label{e:ext}
\E_x\left[ \tau_{D \cap B(z_x, s)}\right] \le C_{\ref{T:exit}.1} \,V(s)V\(\delta_{D}(x)\).
\end{equation}
\item{(2)}
There is a constant $C_{\ref{T:exit}.2}=C_{\ref{T:exit}.2}$ $(\phi, L_0, L_3,$ $ \Lambda, \eta)>0$, such that
for any $s\le R_0/2$,  $\lambda\ge 4$ and $x\in D$ with $\delta_D(x)<\lambda^{-1}s/2$,
\begin{align}\label{e:har}
\p_{x}&\left(Y_{ \tau_{D \cap B(z_x, \lambda^{-1}s)}}\in \{2 \Lambda |\wt y|<y_d, \lambda^{-1}s <|y|<s  \mbox{ in } CS_{z_x}\}\right) \ge C_{\ref{T:exit}.2} \frac{V(\delta_D(x))}{V(s)}.
\end{align}  
\end{description}
\end{thm}

\begin{proof}
Without loss of generality, we assume that $z_x=0$.  
For $R\le R_0/2$,  let $h_R(y)=V(\delta_D(y)){\bf 1}_{D\cap B(0, R)}(y)$.
Let $f\ge 0$ be a smooth radial function such that $f(y)=0$ for $|y|>1$ and $\int_{\R^d} f (y) dy=1$.
For $k\geq 1$,  define $f_k(y):=2^{kd} f (2^k y)$ and $h^{(k)}_{ R}:= f_k*h_R \in C_c^2(\R^d)$,
and let 
$B_k^{\lambda}:=\{y \in D \cap B(0, \lambda^{-1}R/4): \delta_{D \cap B(0, \lambda^{-1}R) }(y) \ge 2^{-k}\}$  for $\lambda\ge 4$.

Since  $h^{(k)}_{R}$ is a $C^2_c$ function, $\gener h^{(k)}_{R}$ is well defined everywhere.
By Lemma \ref{g:st1-1}, for $w\in B_k^{\lambda}$ and $u\in B(0, 2^{-k})$ the following limit
\begin{align*}
 &\lim_{\eps\downarrow 0}\int_{|w-y|>\eps} (h_R(y-u)-h_R(w-u))J(w, y)dy\\
=&\lim_{\eps\downarrow 0}\, \int_{|(w-u)-y^{\prime}|>\eps} (h_R(y^{\prime})-h_R(w-u))J(w, u+y^{\prime})dy^{\prime}= \gener^u h_R(w)
\end{align*}
exists and $ -C_{\ref{g:st1-1}}V(R)^{-1} \le\gener^u h_R(w) \le C_{\ref{g:st1-1}}V(R)^{-1}$.
We note that
\begin{align*}
&\int_{|w-y|>\eps} (h^{(k)}_R(y)-h^{(k)}_R(w))J(w, y)dy\\
=&\int_{|w-y|>\eps}\int_{\R^d}f_k(u)\(h_R(y-u)-h_R(w-u)\)duJ(w, y)dy\\
=&\int_{|u|<2^{-k}}f_k(u)\int_{|w-y|>\eps}\(h_R(y-u)-h_R(w-u)\)J(w, y)dydu.
\end{align*}
By letting $\varepsilon \downarrow 0$ and using the dominated convergence theorem, it follows that for $w \in B_k^{\lambda}$ and $2^{-k+8} <\lambda^{-1}R$,
\begin{align}\label{e:generhk}
&|\gener h^{(k)}_R(w)|=  \left|\int_{|u|<2^{-k}} f_k(u) \gener^u h_R(w)\, du\right| \le C_{\ref{g:st1-1}} V(R)^{-1} \int_{|u|<2^{-k}} f_k(u) \, du =   C_{\ref{g:st1-1}} V(R)^{-1}.
\end{align}

Applying Lemma \ref{Dynkin_formula} to $B_k^{\lambda}$ and $h^{(k)}_{R}$, and using \eqref{e:generhk}, for any $x \in B_k^\lambda$,   we have 
\begin{equation*}\label{e:sh}
\E_{x}\left[h^{(k)}_{R}\big(Y_{\tau_{B_k^\lambda}}\big)\right]-C_{\ref{g:st1-1}}V(R)^{-1}  \E_x\left[\tau_{B_k^\lambda}\right] \le h^{(k)}_{R}(x) \le \E_{x}\left[h^{(k)}_{R}\big(Y_{\tau_{B_k^\lambda}}\big)\right]+C_{\ref{g:st1-1}}V(R )^{-1}  \E_x\[\tau_{B_k^\lambda}\].
\end{equation*}
By letting $k\to \infty$, for any $x \in D \cap B(0, \lambda^{-1}R)$, we obtain 
\begin{align}
&V(\delta_D(x))=h_R(x) \ge \E_{x}\left[h_R\left(Y_{ \tau_{D \cap B(0, \lambda^{-1}R)}} \right)  \right] - C_{\ref{g:st1-1}}  V(R)^{-1} \E_x\left[\tau_{D \cap B(0, \lambda^{-1}R)}\right]\label{e:ss}\\ 
\mbox{ and }\,\, &V(\delta_{D}(x))=h_R(x)  \le   \E_{x}\left[h_R\left(Y_{ \tau_{D \cap B(0, \lambda^{-1}R)}} \right)  \right]+ C_{\ref{g:st1-1}}  V(R)^{-1}\E_x\left[\tau_{D \cap B(0, \lambda^{-1}R)}\right].\label{e:ss1}
\end{align}

For any $z  \in D \,\cap\, B(0, \lambda^{-1}R)$ and $y \in D\,\cap \,(B(0, R) \setminus B(0, \lambda^{-1}R))$, 
by the fact that $\nu$ is decreasing and  \eqref{nu1},
$\nu(|y-z| ) \ge  \nu(2|y|) \ge  c_1 \nu(|y|)$.
So by \eqref{a:kappa}, {\bf(J1.1)} and \eqref{e:levy}, we obtain 
\begin{align}\label{e:30} 
\E_x\left[h_R\left(Y_{\tau_{D\cap B(0, \lambda^{-1}R)}}\right)\right]
&\ge  L_0^{-1}\E_x \int_{ D\cap \(B(0,R)\setminus B(0, \lambda^{-1}R)\)} \int_0^{\tau_{D \cap B(0, \lambda^{-1}R)}}  \nu(|Y_t-y|)dt h_R(y) dy \nonumber \\
&\ge  L_0^{-1} \, c_1  \, \E_x\left[\tau_{D \cap B(0, \lambda^{-1}R)}\right] \int_{  D\cap \(B(0, R)\setminus B(0, \lambda^{-1}R)\)}   \nu(|y|)h_R(y) dy.
\end{align}

Let $A:=\{(\wt{y}, y_d): 2 \Lambda |\wt y|<y_d\}$. 
For any $y\in A \cap B(0, R)$, 
since $y_d>2 \Lambda |\wt y|>2\Lambda |\wt{y}|^{2}>\varphi(\wt y)$,    we have $A \cap B(0, R)\subset D\cap B(0,R)$ and 
 $$
\delta_D(y)\ge (1+\Lambda)^{-1}\left(y_d-\varphi(\wt y)\right) \ge  (2\Lambda)^{-1} (y_d -\Lambda|\wt y|)> (4\Lambda)^{-1}y_d \ge  (4\Lambda((2\Lambda)^{-2}+1)^{1/2})^{-1}|y|.
$$
Combining this and \eqref{V:sc}, 
$V(\delta_D(y))\ge c_2V(|y|)$. 
By changing to polar coordinates with $|y|=t$, \eqref{e:com} and Proposition \ref{VbisEst}, we obtain that
\begin{align} \label{e:31}
\int_{  D\cap \(B(0, R)\setminus B(0, \lambda^{-1}R)\)}   & \nu(|y|)h_R(y) dy 
\ge c_2\int_{ A\cap \(B(0, R)\setminus B(0, \lambda^{-1}R)\)} \nu(|y|)V(|y|) dy\nonumber\\
&\ge c_3\int_{\lambda^{-1}R}^{R} \nu(t) V(t) t^{d-1}dt \ge c_4\int_{\lambda^{-1}R}^{R}V(t)^{-1}t^{-1}dt\nn\\
&\ge c_4\cdot C_{\ref{VbisEst}}\int_{\lambda^{-1}R}^{R}\frac{V^{\prime}(t)}{V(t)^2}dt
=  c_4\cdot C_{\ref{VbisEst}}\(V(\lambda^{-1}R)^{-1}-V(R)^{-1}\).
\end{align}

Combining \eqref{e:30} and \eqref{e:31}, there exists $C_5:=C_5(\lC, \ua,  \uC, L_0, \Lambda, d)>0$ such that \begin{align}\label{e:32}
\E_x\left[h_r\left(Y_{\tau_{D\cap B(0, \lambda^{-1}R)}}\right)\right]\ge C_5\,  \E_x\left[\tau_{D \cap B(0, \lambda^{-1}R)}\right]\(V(\lambda^{-1}R)^{-1}-V(R)^{-1}\).
\end{align}
Using \eqref{V:sc} again,  $V(\lambda^{-1}R)\le V(\lambda_0^{-1}R)\le C_V\lambda_0^{-\la/2}V(R)$ for any $\lambda\ge \lambda_0\ge 4$.
Let $\lambda_0:=(2C_V({C_5}+C_{\ref{g:st1-1}})/C_5)^{2/\la}\ge 1$.
Then combining \eqref{e:ss} and  \eqref{e:32},  we have that for $\lambda \ge \lambda_0$  
\begin{align}
V(\delta_{D}(x))
\ge & \left(C_5V( \lambda^{-1}R)^{-1} -(C_5+C_{\ref{g:st1-1}})V(R)^{-1} \right) \E_x[\tau_{D \cap B(0, \lambda^{-1}R)}]\label{e:nnVE}\\
\ge & \,(C_5/2)\, V(\lambda^{-1}R)^{-1}\E_x[\tau_{D \cap B(0, \lambda^{-1}R)}]\label{e:nVE}.
\end{align}
Thus, we have proved \eqref{e:ext} with $A_{\ref{T:exit}}=\lambda_0^{-1}$ and $s=\lambda^{-1}R$ where $\lambda\ge \lambda_0$.

From \eqref{e:ss1} and Corollary  \ref{st5} with \eqref{e:com} for $\delta_D(x)< \lambda^{-1}R/2$ and $\lambda\ge 4$, we first note that
\begin{align}\label{e:39}
V(\delta_D(x))
 \le &\,\, V(R)\,\p_{x}\left(Y_{ \tau_{D \cap B(0, \lambda^{-1}R)}}\in D \right)+ C_{\ref{g:st1-1}}  V(R)^{-1}\E_x\left[\tau_{D \cap B(0, \lambda^{-1}R)}\right]\nn\\
\le &\,\, \left( C_{\ref{st5}}V(R)\,V(\lambda^{-1}R)^{-2}+ C_{\ref{g:st1-1}}  V(R)^{-1}\right)\E_x\left[\tau_{D \cap B(0, \lambda^{-1}R)}\right]\nn\\
=&\,\,  c_6 V(R)\left(V(\lambda^{-1}R)^{-2}+V(R)^{-2}\right)\E_x\left[\tau_{D \cap B(0, \lambda^{-1}R)}\right].
\end{align}
{By \eqref{V:sc} and the subadditivity of $V$,
$V(s)^{-1}\ge 3^{-1}(3\lambda^{-1}R/s)^{\la/2}C_V^{-1} V(\lambda^{-1}R)^{-1}$  for $s\le 3\lambda^{-1}R$.} 
Combining this with  \eqref{e:com} and using  the polar coordinate with $|y|=t$, we have that
\begin{align} \label{e:33}
\int_{A\cap \( B(0, 3\lambda^{-1}R)\setminus B(0, 2\lambda^{-1}R)\)} \nu(|y|)  dy
&\ge c_7 \int_{2\lambda^{-1}R}^{3\lambda^{-1}R} \nu(t) t^{d-1}dt\nn\\
&\ge c_8 \int_{2\lambda^{-1}R}^{3\lambda^{-1}R} V(t)^{-2}t^{-1}dt
\ge c_9 V(\lambda^{-1}R)^{-2}.
\end{align}
For any  $z\in B(0, \lambda^{-1}R)$ and $y\in B(0, 3\lambda^{-1}R)\setminus B(0, 2\lambda^{-1}R)$, by \eqref{nu1} and 
the fact that $\nu$ is decreasing, 
$\nu(|y-z| ) \ge \nu(|y|+|z|) \ge  \nu(3|y|/2) \ge  c_{10} \nu(|y|)$.
So by \eqref{a:kappa}, {\bf(J1.1)} \eqref{e:levy} and \eqref{e:33}, we obtain 
\begin{align}\label{e:40}
&\p_{x}\left(Y_{ \tau_{D \cap B(0, \lambda^{-1}R)}}\in A\cap \left(B(0, 3\lambda^{-1}R)\setminus B(0, 2\lambda^{-1}R)\right)\right)\nn\\
\ge &\,L_0^{-1}\E_x\left[\int_0^{\tau_{D \cap B(0, \lambda^{-1}R)}}\int_{ A\cap \left(B(0, 3\lambda^{-1}R)\setminus B(0, 2\lambda^{-1}R)\right)}\nu(|Y_s-y|)dy \, ds\right]\nn\\
\ge &\, L_0^{-1}c_{10}\E_x\left[\int_0^{\tau_{D \cap B(0, \lambda^{-1}R)}}\int_{ A\cap \left(B(0, 3\lambda^{-1}R)\setminus B(0, 2\lambda^{-1}R)\right)}\nu(|y|)dy \, ds\right]\nn\\
\ge &\, c_9 \, c_{10} L_0^{-1} V(\lambda^{-1}R)^{-2}\E\left[\tau_{D \cap B(0, \lambda^{-1}R)}\right].
\end{align}
Hence combining \eqref{e:39}, \eqref{e:40}, and the fact that $V$ is increasing, we conclude that for $\lambda\ge 4$, 
\begin{align*}
V(\delta_D(x))\le &
c_{11} V(R)\left(V(\lambda^{-1}r)^{-2}+V(R)^{-2}\right)V(\lambda^{-1}R)^2\\
&\qquad\cdot \p_{x}\left(Y_{ \tau_{D \cap B(0, \lambda^{-1}R)}}\in A\cap \left(B(0, 3\lambda^{-1}R)\setminus B(0, 2\lambda^{-1}R)\right)\right)\nn\\
\le& 2c_{11}V(R)\p_{x}\left(Y_{ \tau_{D \cap B(0, \lambda^{-1}R)}}\in A\cap \left(B(0, R)\setminus B(0, \lambda^{-1}R)\right)\right).
\end{align*}
Thus, we have proved \eqref{e:har} with $s=R$.

\begin{remark}

\end{remark}

\section{Upper bound estimates}\label{sec:upbd}

In this section, we derive the upper bound estimate on $p_D(t,x, y)$ for $t\le T$ in $C^{1, 1}$  open  set $D$ with $C^{1, 1}$ characteristics $(R_0, \Lambda)$ .
As before, we will assume that $R_0<1$ and $\Lambda>1$ and fix such $C^{1, 1}$  open  set $D$ throughout this section.
We first introduce the next lemma which give a guideline to obtain the upper bound estimate on $p_D(t, x, y)$ (for its proof, see   \cite[Lemma 3.1]{CKS1} and \cite[{Lemma 1.10}]{BGR2013_3}).
Applying \eqref{e:ext} and Theorem \ref{T:2.1} to \eqref{eq:ub1}, in Proposition \ref{P:up1} we will obtain the intermediate upper bound for $p_D(t,x, y)$ having one boundary decay. 
Applying this result, \eqref{e:ext} and the upper bound of the survival probability of $Y$ (Lemma \ref{L:exit2}) to \eqref{eq:ub0}, we can get the short time sharp upper bound estimate for $p_D(t, x, y)$.

\begin{lem}\label{L:4.1} 
Let $Y$ be a symmetric pure jump Hunt process whose jumping intensity kernel $J$ satisfies the conditions {\bf(J1.1)} and {\bf(J1.2)}.
Suppose that $E\subset \R^d$ be an open set. Let $U_1, U_3\subset E$ be disjoint open subsets and $U_2:=E\setminus(U_1\cup U_3)$.
If $x\in U_1$, $y\in U_3$ and $t>0$, we have
\begin{align}
p_E(t, x, y) &\le\,\, \mathbb{P}_x\left(Y_{\tau_{U_1}}\in U_2\right) \cdot\sup_{s<t, z\in U_2} p_E(s, z, y)\nn\\
&\qquad + \int_0^t \p(\tau_{U_1}>s)\p_y(\tau_E>t-s) \, ds\cdot \sup_{u\in U_1, z\in U_3}J(u,z)\label{eq:ub0}\\
&\le \,\,\mathbb{P}_x\left(Y_{\tau_{U_1}}\in U_2\right) \cdot\sup_{s<t, z\in U_2} p(s, z, y)+\left(t \wedge \mathbb{E}_x[\tau_{U_1}]\right)\cdot\sup_{u\in U_1, z\in U_3} J(u,z). \label{eq:ub1}
\end{align}
\end{lem}

For the remainder of the section, we assume that $Y$ is the symmetric pure jump Hunt process with the jumping intensity kernel $J$ satisfying the conditions {\bf(J1.1)}, {\bf(J1.2)},  \textbf{(K$_{\eta}$)} and  \textbf{(SD)}.
Let
$$a_{T, R_0}:=[V(A_{\ref{T:exit}}R_0/4)]^2T^{-1}$$ 
where $A_{\ref{T:exit}}$ is the constant in Theorem \ref{T:exit}(1).
Denote $V^{-1}$ be the inverse function of $V$, then $V^{-1}(\sqrt{a_{T, R_0}\cdot t})\le A_{\ref{T:exit}}R_0/4$  for any $t\le T$.
\begin{lem}\label{L:exit2}
There exists $C_{\ref{L:exit2}}=C_{\ref{L:exit2}}(\phi, L_0, L_3,  \eta, R_0,  \Lambda, T)$ such that for any $t\le T $ and $x\in D$, we have that
\begin{equation*}
\p_x(\tau_D>t)\le C_{\ref{L:exit2}}\left(1\wedge \frac{V(\delta_D(x))}{ \sqrt{t}}\right).
\end{equation*}
\end{lem}
\proof
Let  $r_t:=V^{-1}( \sqrt  { a_{T, R_0}\cdot t}\,)$.
We only consider the case $V(\delta_D(x))<    \sqrt  {a_{T, R_0}\cdot t}$ which implies $\delta_D(x)<  A_{\ref{T:exit}}R_0/4<R_0/4$.
Let $U_t:=D\cap B(z_x, 2r_t)\subset D$ where $z_x\in\partial D$ with $\delta_D(x)=|x-z_x|$.
Then, using Chebyshev's inequality and Corollary \ref{st5},
we first obtain that
\begin{align*}
\p_x(\tau_D>t)
&=\p_x(\tau_{U_t}>t, \tau_D=\tau_{U_t})+\p_x(\tau_D >\tau_{U_t}>t)\\
&\le \p_x(\tau_{U_t}>t)+\p_x(Y_{\tau_{U_t}}\in D)\le \left(t^{-1}+C_{\ref{st5}}/\phi(2 r_t)\right) \E_x \tau_{U_t}.
\end{align*}
From \eqref{e:com} and the fact that $V$ is increasing and subadditive, 
$$\phi(2 r_t)\asymp [V(2 r_t)]^2\asymp [V( V^{-1}(\sqrt{a_{T, R_0}\cdot t}))]^2=a_{T, R_0}\cdot t.$$
Therefore, using \eqref{e:ext} in Theorem \ref{T:exit}, we conclude that
\begin{align*}
\p_x(\tau_D>t) \le c_1\frac{1}{t} \E_x \tau_{U_t} \le c_1 C_{\ref{T:exit}.1}\frac{1}{t} V(r_t)V(\delta_D(x))\le c_2 \frac{V(\delta_D(x))}{\sqrt{t}}.
\end{align*}
\qed

We will use the following inequality several times, which follows from  {\bf(WS)}:
 there exist $C_I:=(\uC\lC^{-1}\vee \lC^{-2})^{1/\la}>1$ such that
\begin{align}\label{i:sc}
C_I^{-1}\left(\frac{r}{R}\right)^{1/\la}\le \frac{\phi^{-1}(r)}{\phi^{-1}(R)}\le C_I \left(\frac{r}{R}\right)^{1/\ua}& \qquad \mbox{ for }\,\, 0<r\, \le  R. 
\end{align}

Recall the functions $F_{a, \gamma, T}(t, r)$ 
and $\Psi(t,x)$ 
 are defined in \eqref{eq:qd} and \eqref{e:dax}, respectively.

\begin{prop}\label{P:up1}

Let $a\le a_{T, R_0}$.
\begin{description}
\item[(1)]
Suppose that $D$ is bounded and the jumping intensity kernel $J$ satisfies the condition {\bf(J1)}. Then there exists a positive constant $C_{\ref{P:up1}.1}=C_{\ref{P:up1}.1}({ \phi, L_0, L_3,  \eta, R_0,  \Lambda, T, \diam(D), {a}})$ such that for any $(t, x, y)\in (0, T]\times D\times D$ with $ V^{-1}( \sqrt{a\cdot t})\le |x-y|$, we have 
\begin{equation*}
p_D(t,x, y)\le C_{\ref{P:up1}.1}\Psi(t,x)\cdot 
\left([\phi^{-1}(t)]^{-d}\wedge  t\nu(|x-y|)\right).
\end{equation*}

\item[(2)]
Suppose that the jumping intensity kernel $J$ satisfies the condition {\bf(J2)}. Then there exists a positive constant $C_{\ref{P:up1}}=C_{\ref{P:up1}}(\beta, \phi, L_0, L_3,  \eta, R_0,  \Lambda, T, {a})$ such that 
for any $(t, x, y)\in (0, T]\times D\times D$ with 
{$|x-y|\ge  V^{-1}( \sqrt{a\cdot t})\cdot{\bf 1}_{\beta\in[0, 1]} 
+2\cdot{\bf 1}_{\beta\in(1,\infty)}+\left(2+  V^{-1}( \sqrt{a\cdot t})\right)\cdot {\bf 1}_{\beta= \infty}$}, 
we have 
\begin{equation*}
p_D(t,x, y)\le C_{\ref{P:up1}}\Psi(t,x)\cdot 
\begin{cases}
F_{ C_{\ref{T:2.1}}\wedge\gamma_1, \gamma_1, T}\left(t, |x-y|/{3}\right) &\mbox{ if } \beta\in[0,\infty),\\
\left({2t}/{T|x-y|}\right)^{C_{\ref{T:2.1}}|x-y|/2}
&\mbox{ if } \beta=\infty.
\end{cases}
\end{equation*}

\end{description}
\end{prop}

\proof
Since we assume that $D$ is bounded in (1), by applying Theorem \ref{T:n2.1} instead of Theorem \ref{T:2.1}
the proof of (1) is similar to the that of (2), so  we only give the proof of (2).
Let $r_t:=V^{-1}(\sqrt{a\cdot t})/{9}$.
If $\delta_D(x)\ge r_t/2$, using subadditivity of $V$, we see that $\Psi(t, x)\asymp 1$. Thus, by Theorem \ref{T:2.1}, and the fact that $r\to 
F_{c, \gamma, T}
(t, r)$ is decreasing, we obtain the conclusion.

Let $0<\delta_D(x)\le r_t/2$.
Since $9r_t\le V^{-1}(\sqrt{a_{T , R_0}\cdot T})= { A_{\ref{T:exit}} R_0/4<1}$, $|x-y|\ge 9r_t$ for all $\beta \in  [0, \infty]$.
{Choose a point $z_x\in\partial D$ with $\delta_D(x)=|x-z_x|$
and let} $U_1:=B(z_x, r_t)\cap D$, $U_3:=\{z\in D:|z-x|\ge |x-y|/2\}$ and $U_2:=D\setminus (U_1\cup U_3)$. Then $x\in U_1$, $y\in U_3$ and $U_1\cap U_3=\emptyset$. Note that 
$|x-y|/2\le |x-y|-|z-x|\le |y-z|$ for any  $z\in U_2$. 
Therefore,  by virtue of Theorem \ref{T:2.1}, we have
we obtain 
\begin{align}
\sup_{s<t, z\in U_2} p(s, z, y)\le C_{\ref{T:2.1}} \sup_{s<t,|z-y|>|x-y|/2} F_{ C_{\ref{T:2.1}}, \gamma_1, T}(s, |z-y|) \le c_1 F_{ C_{\ref{T:2.1}}, \gamma_1, T}(t, |x-y|/2)\label{eq:u1}.
\end{align}
In fact, if $\beta\in (1, \infty]$,  we have $|z-y|\ge|x-y|/2>1$ and so $F_{C_{\ref{T:2.1}}, \gamma_1, T}(s, |z-y|)$ is increasing in $s$.
If $\beta\in [0, 1]$, we have $|z-y|\ge|x-y|/2\ge  V^{-1}( \sqrt{a\cdot t})/2$
and so $\Big([\phi^{-1}(s)]^{-d}\wedge s\nu(|z-y|)e^{-\gamma |z-y|^\beta}\Big)\asymp s\nu(|z-y|)e^{-\gamma |z-y|^\beta}$ using \eqref{e:com}, \eqref{V:sc} and \eqref{i:sc}. Also, $s\nu(r)e^{-\gamma r^\beta}$ is increasing in $s$.
Thus, combining there observations with the fact $r\to F_{C_{\ref{T:2.1}}, \gamma_1, T}(t, r)$ is decreasing, the second inequality above holds.

By Corollary \ref{st5} and \eqref{e:ext} in Theorem \ref{T:exit}, we obtain 
\begin{align}\label{eq:u2}
\p_x(Y_{\tau_{U_1}}\in U_2)& \le \p_x\left(Y_{\tau_{U_1}}\in D\right)\le C_{\ref{st5}}\phi(r_t)^{-1}\E_x \tau_{U_1}\nn\\
&\le C_{\ref{st5}}\cdot C_{\ref{T:exit}.1} V(\delta_D(x))/V(r_t)\le c_3 {V(\delta_D(x))}/{\sqrt t}.
\end{align}
In the last inequality we use  monotonicity and subadditivity of $V$, which imply $V(r_t)\asymp \sqrt{t}$.

Note that for $u\in U_1$ and $z\in U_3$ that 
\begin{align}\label{eq:u211}|u-z|\ge |z-x|-|x-z_x|-|u-z_x|\ge |x-y|/2-3{r_t}/2.\end{align}
Let $\beta\in [0, \infty)$. 
Since $|x-y|\ge 9{r_t}$, from  \eqref{eq:u211} we have $|u-z|\ge |x-y|/3$ for  $(u, z) \in U_1 \times U_3$, therefore by \eqref{a:kappa}, \eqref{e:Exp} and {\bf(J2)}, 
$$\left(\sup_{u\in U_1, z\in U_3} J(u, z)\right)\le c_4 e^{-\gamma_1(|x-y|/3)^{\beta}}
\nu(|x-y|)
\le c_5 t^{-1} F_{ \gamma_1, \gamma_1, T}(t, |x-y|/3).$$
Combining this with \eqref{e:ext} in Theorem \ref{T:exit}, we conclude that
\begin{align}\label{eq:u3}
\E_x [\tau_{U_1}]\left(\sup_{u\in U_1, z\in U_3} J(u, z)\right)&\le C_{\ref{T:exit}.1}V(r){V(\delta_D(x))}\cdot c_5 t^{-1}F_{ \gamma_1, \gamma_1, T}(t, |x-y|/3)\nn\\
&\le c_6 \frac{V(\delta_D(x))}{\sqrt{t}} F_{ \gamma_1, \gamma_1, T}(t, |x-y|/3).
\end{align}
If $\beta=\infty$, since $|u-z| \ge |x-y|/2-3{r_t}/2\ge (1+4{r_t})-3{r_t}/2\ge 1$, we have $J(u, z)=0$.

Hence, by applying \eqref{eq:u1}--\eqref{eq:u3} to \eqref{eq:ub1} for the case $\beta\in[0, \infty)$ and by applying \eqref{eq:u3} to \eqref{eq:ub1} for the case $\beta=\infty$, we reach the conclusion.
\qed

We denote by $X$ the process in the case $\beta=0$ in {\bf(J2)}, that is, $X$ is a symmetric Hunt process whose jumping kernel is $J^X(x, y):= \kappa(x,y)\nu(|x-y|).$
\begin{prop}\label{P:X}
\begin{description}
\item[(1)]
Suppose that the jumping intensity kernel $J$ satisfying {\bf(J1)} and $D$ is bounded.
There exists a positive constant $C_{\ref{P:X}.1}=C_{\ref{P:X}.1}({\phi, L_0, L_3,  \eta, R_0,  \Lambda, T, \diam(D)})$
such that for any $(t, x, y)\in (0, T]\times D\times D$, we have 
\begin{align*}
p_D(t, x, y)\le C_{\ref{P:X}.1}\Psi(t,x)\Psi(t,y)\left([\phi^{-1}(t)]^{-d}\wedge t\nu(|x-y|)\right).
\end{align*}
\item[(2)]
There exists a positive constant $C_{\ref{P:X}}=C_{\ref{P:X}}({\phi, L_0,  L_3,   \eta, R_0,  \Lambda, T})$ such that for any $(t, x, y)\in (0, T]\times D\times D$, we have  
\begin{align*}
p_D^X(t, x, y)\le C_{\ref{P:X}}\Psi(t,x)\Psi(t,y)\left([\phi^{-1}(t)]^{-d}\wedge t\nu(|x-y|)\right).
\end{align*}
\end{description}
\end{prop}

\proof
Using Theorem  \ref{T:n2.1} and  Proposition~\ref{P:up1} (1)
instead of Theorem~\ref{T:2.1} and  Proposition~\ref{P:up1} (2) respectively,  
the proof of (1) is almost identical to the one of (2). So we only give the proof of (2).

The semigroup property, Theorem~\ref{T:2.1} (for $\beta=0$), \eqref{i:sc} and Lemma~\ref{L:exit2} yield
\begin{align*}
p_D^X(t/2, x, y)& \le \left(\sup_{z, w\in D} p_D^X(t/4, z, w)\right)\int_D p_D^X(t/4, x, z)dz
\\
&\le \,c_1[\phi^{-1}(t/4)]^{-d} \p_x(\tau_D >t/4)  \le c_2 [\phi^{-1}(t)]^{-d} \Psi(t,x).
\end{align*}
Thus, by Proposition~\ref{P:up1}, \eqref{i:sc} and Theorem~\ref{T:2.1} (for $\beta=0$), we obtain 
$$p_D^X(t/2, x, y) \le c_3 \Psi(t/2,x) \left([\phi^{-1}(t/2)]^{-d}\wedge (t/2)\nu(|x-y|)\right) \le  c_4 \Psi(t,x) p^X(t/2, x, y).$$
Combining these with Theorem~\ref{T:2.1} (for $\beta=0$), the symmetry $p_D^X$ and the semigroup property of $p^X$, we conclude that 
\begin{align*}
p_D^X(t, x, y)\,&=\, \int_D p_D^X(t/2, x, z)\cdot p_D^X(t/2, z, y)dz
\le \, c_4^2 \Psi(t,x) \Psi(t,y) p^X(t, x, y)\\
&\,\le \, c_5 \Psi(t,x) \Psi(t,y) \left([\phi^{-1}(t)]^{-d}\wedge t\nu(|x-y|)\right). 
\end{align*}
\qed

Suppose that the jumping intensity kernel $J$ satisfying {\bf(J2)}.
By Meyer's construction (e.g., see \cite[\S 4.1]{MR2357678}), when $\beta \in (0, \infty]$
the process $Y$ can be constructed from $X$ by removing jumps of size greater than $1$ with suitable rate.
Let $p_D^X(t,x,y)$ be the transition density function of  $X$ on $D$. 
For $\beta \in (0, \infty]$, we define
\begin{align*}
{\cal J}  (x):=\int_{\R^d} \kappa(x, y)\nu(|x-y|)\left(1- \chi(|x-y|)^{-1}\right)dy
\end{align*}
where $\chi(|x-y|)$ is defined in \eqref{e:Exp}.
Then $\|{\cal J}\|_{\infty}\le c_1\int_{|z|\ge 1} \nu(|z|)dz <\infty$. 
By \cite[Lemma 3.6]{BBCK} we have 
\begin{align}
p_D(t,x,y)\,\le\, e^{T\|\cal J\|_{\infty}} p_D^X(t,x,y) \quad \mbox{ for any } (t,x,y) \in (0, T] \times D \times D. \label{e:u.0}
\end{align}
Thus, the sharp upper bound of of $p_D(t,x,y)$ for  $|x-y|< M$ for some $M>0$ follows from the one of  $p_D^X(t,x,y)$  and \eqref{e:u.0}.
Therefore Combining \eqref{e:u.0},  Propositions~\ref{P:up1}(2) and~\ref{P:X}(2), we have the following result.

\begin{prop}\label{P:up2}
Suppose that the jumping intensity kernel $J$ satisfying {\bf(J2)}.
There exists a positive constant $C_{\ref{P:up2}}=C_{\ref{P:up2}}(\beta, \phi, L_0, L_3,   \eta, R_0,  \Lambda, T)$ such that for every $(t, x, y)\in (0, T]\times D\times D$ we have 
\begin{align}
p_D(t, x, y)\le C_{\ref{P:up2}} \Psi(t,x)\cdot 
\begin{cases}
F_{ C_{\ref{T:2.1}}\wedge\gamma_1, \gamma_1, T}(t, |x-y|/{3})&\mbox{ if } \beta\in[0,\infty),\\
F_{ C_{\ref{T:2.1}}, \gamma_1, T}(t, |x-y|/2)&\mbox{ if } \beta=\infty,
\end{cases}
\end{align}
where $C_{\ref{T:2.1}}$ is the constant in Theorem~\ref{T:2.1} and $\gamma_1$ is the constant in \eqref{e:Exp}.
\end{prop}

Now we are ready to prove the upper bound of Theorem \ref{t:nmain}(1) and Theorem \ref{t:main}(1).

\medskip

\noindent {\bf Proofs of the upper bounds of $p_D(t,x, y)$ in Theorems \ref{t:nmain}(1) and \ref{t:main}(1)}. 
In Proposition \ref{P:up2}(1), we have proved the upper bound of $p_D(t,x, y)$ in Theorem \ref{t:nmain}(1). So we only give the proof of the upper bound of $p_D(t,x, y)$ in Theorem \ref{t:main}(1).

Let  $r_t:=V^{-1}(\sqrt{a_{T, R_0}\cdot t})$ so that $r_t\le A_{\ref{T:exit}}R_0/4<1/4$.
By Proposition \ref{P:up2}  and the symmetry of $p_D(t,x, y)$, 
we only need to prove the upper bound of $p_D(t,x, y)$ for the case  $\delta_D(x)\vee \delta_D(y)<r_t$, which we will assume throughout the proof. 

If $\beta = \infty$ and ${6} <|x-y|\le  {6}(1 \vee C_{\ref{T:2.1}}^{-1})$, by \eqref{e:u.0} and Proposition~\ref{P:X},  we have 
\begin{align*}
p_D(t, x, y)\le c_1 \Psi(t,x)\Psi(t,y) (t/T) \le c_1 \Psi(t,x)\Psi(t,y) (t/T)^{(C_{\ref{T:2.1}} \wedge 1)|x-y|/6}.
\end{align*}
If either the case $\beta\in[0, \infty)$  and $|x-y|\le  {6}(1 \vee C_{\ref{T:2.1}}^{-1})$ holds or the case $\beta = \infty$ and $|x-y|\le  {6}$ holds, by \eqref{e:u.0} and Proposition~\ref{P:X}(2),  we have
\begin{align*}
p_D(t, x, y)\le e^{T\|{\cal J}\|_{\infty}} p_D^X(t,x,y) \le c_2 \Psi(t,x)\Psi(t,y) \left([\phi^{-1}(t)]^{-d}\wedge t\nu(|x-y|)\right).
\end{align*}
Thus,  the upper bound of $p_D(t,x, y)$ in Theorem \ref{t:main}(1) holds for $|x-y|\le  {6}(1 \vee C_{\ref{T:2.1}}^{-1})$.

For the remainder of the proof, we assume that $\delta_D(x)\vee \delta_D(y)<r_t$ and $ |x-y|>  {6}(1 \vee C_{\ref{T:2.1}}^{-1}) $.
For any $x$ with $\delta_D(x)<r_t$, let $z_x\in \partial D$ such that $\delta_D(x)=|z_x-x|$. Let $U_1:=B(z_x, r_t)\cap D$, $U_3:=\{z\in D: |z-x|\ge|x-y|/2\}$, and $U_2:=D\setminus (U_1\cup U_3)$.
Note that $x\in U_1$ and $y\in U_3$ and $|x-y|/2\le |z-y|$ for $z\in U_2$.
Thus, by Proposition~\ref{P:up2} we have   
\begin{align}\label{eq:p_D}
&\sup_{s<t, z\in U_2} p_D(s,z,y)\nn\\
\le&\sup_{s<t, z\in U_2} C_{\ref{P:up2}} \frac{V(\delta_D(y))}{\sqrt{s}}\cdot 
\left(F_{ C_{\ref{T:2.1}}\wedge \gamma_1, \gamma_1, T}(s, |z- y|/{3})\cdot{\bf 1}_{\beta\in[0, \infty)}
+F_{C_{\ref{T:2.1}}, \gamma_1, T}(s, |z- y|/2)\cdot{\bf 1_{\beta=\infty}}\right)\nn\\ 
\le & \,C_{\ref{P:up2}} \,V(\delta_D(y)) \sup_{s<t, |x-y|/2\le |z-y|}\frac{1}{\sqrt{s} }\nn\\
&\qquad \qquad \qquad \cdot  
\left(F_{ C_{\ref{T:2.1}}\wedge \gamma_1, \gamma_1, T}(s, |z- y|/{3})\cdot{\bf 1}_{\beta\in[0, \infty)}
+F_{C_{\ref{T:2.1}}, \gamma_1, T}(s, |z- y|/2)\cdot{\bf 1}_{\beta=\infty}\right)\nn\\ 
\le& \, c_3\,\frac{V(\delta_D(y))}{\sqrt{t}}\cdot  \left(F_{ C_{\ref{T:2.1}}\wedge \gamma_1, \gamma_1, T}(t, |x- y|/{6})\cdot{\bf 1}_{\beta\in[0, \infty)}
+F_{C_{\ref{T:2.1}}, \gamma_1, T}(t, |x- y|/4)\cdot{\bf 1_{\beta=\infty}}\right).
\end{align}
The last inequality is clear for $\beta \in [0, \infty)$ by the definition of  $F_{C_{\ref{T:2.1}}\wedge \gamma_1, \gamma_1, T}(t,r)$ and for $\beta=\infty$  we used the fact that $s \to s^{-1/2} (s/Tr)^{ar}$ is increasing if $ar \ge 1$. 
Hence from \eqref{eq:u2} and \eqref{eq:p_D}, we obtain
\begin{align}\label{main:I}
&\mathbb{P}_x\left(Y_{\tau_{U_1}}\in U_2\right)\Big(\sup_{s<t, z\in U_2} p_D(s, z, y)\Big)\nn\\
&\le c_4 \frac{V(\delta_D(x))}{\sqrt{t}} \frac{V(\delta_D(y))}{\sqrt{t}} \cdot
\begin{cases}
F_{ C_{\ref{T:2.1}}\wedge\gamma_1, \gamma_1, T}(t, |x-y|/{6})&\mbox{ if } \beta\in[0,\infty),\\
F_{ C_{\ref{T:2.1}}, \gamma_1, T}(t, |x-y|/{4})&\mbox{ if } \beta=\infty.
\end{cases}
\end{align}
Also from Lemma~\ref{L:exit2}, we have 
\begin{align}\label{main:II-1}
&\int_0^t \mathbb{P}_x(\tau_{U_1}>s) \mathbb{P}_y(\tau_D>t-s) ds \le \int_0^t \mathbb{P}_x(\tau_{D}>s) \mathbb{P}_y(\tau_D>t-s) ds \nn\\
&\qquad\le C_{\ref{L:exit2}}^2 \, V(\delta_D(x)) V(\delta_D(y))\int_0^t s^{-1/2}  (t-s)^{-1/2} ds \le c_5 \, t \,\frac{V(\delta_D(x))}{\sqrt{t}} \frac{V(\delta_D(y))}{\sqrt{t}}.
\end{align}
For $(u,z) \in U_1 \times U_3$ and $|x-y|>{6>6r_t}$, note that $|u-z|{\ge |x-y|-|x-u|-|z-y|\ge |x-y|/3}$.
Thus, if $\beta\in[0,\infty)$, by \eqref{e:Exp} and {\bf(J2)},
$$\left(\sup_{u\in U_1, z\in U_3} J(u, z)\right)\le 
c_6 e^{-\gamma_1(|x-y|/{3})^{\beta}}
\nu(|x-y|/3)
\le c_7 t^{-1} F_{ \gamma_1, \gamma_1, T}(t, |x-y|/{3}).$$
Combining this with \eqref{main:II-1}, we obtain 
\begin{align}\label{main:II}
&\int_0^t \mathbb{P}_x\left(\tau_{U_1}>s\right) \mathbb{P}_y \left(\tau_D>t-s \right)ds \cdot \left(\sup_{u\in U_1, z\in U_3} J(u, z)\right)\nn\\
&\le c_{8}\frac{V(\delta_D(x))}{\sqrt{t}}\frac{V(\delta_D(y))}{\sqrt{t}}F_{ \gamma_1, \gamma_1,T}(t, |x-y|/{3}).
\end{align}
If $\beta=\infty$, since $|u-z|> 1$, $J(u, z)=0$ on $ U_1 \times U_3$.

Therefore by applying \eqref{main:I} and \eqref{main:II} for $\beta\in[0, \infty)$ and by applying \eqref{main:I} for $\beta=\infty$  in \eqref{eq:ub0} of Lemma~\ref{L:4.1}, we prove the upper bound of  $p_D(t,x, y)$ in Theorem~\ref{t:main}(1)
 for  $\delta_D(x)\vee \delta_D(y)<r_t$ and $ |x-y|> {6}(1 \vee C_{\ref{T:2.1}}^{-1})$.
\end{proof}

\section{Preliminary lower bound estimates }\label{sec:plbd}

In this section, we discuss a preliminary lower bound for $p_D(t, x, y)$. In this section we will always assume that $Y$ is the symmetric pure jump Hunt process with the jumping intensity kernel $J$ satisfying either the conditions {\bf(J1.2)} and  {\bf(J1.3)} or the condition {\bf(J2)}.

Since $Y$ satisfies conditions imposed  in \cite{MR2524930}, using \cite[Theorem 5.2 and Lemma 2.5]{MR2524930}, the proof of the next lemma is the same as that of \cite[Lemma 3.2]{MR3237737}. Thus, we omit the proof. 

\begin{lem}\label{L:plbd1}
Let $a$,$b$ and $T$ be positive constants.
Then 
there exists a constant $C_{\ref{L:plbd1}}=C_{\ref{L:plbd1}}$ $(a, b, {L_0, \phi}, T)>0$ such that for all $\lambda \in (0, T]$ we have
\begin{align}
\inf_{y\in\R^d \atop |y -z| \le b \phi^{-1}(\lambda)} \p_y \left(\tau_{B(z, 2b \phi^{-1}(\lambda) )} > a\lambda \right) \ge C_{\ref{L:plbd1}}.
\end{align}
\end{lem}

Let $D$ be an arbitrary non-empty open set, and $a$ and $T$ be positive constants.
We use the convention that  $\delta_D(\cdot)\equiv \infty$ when $D=\Rd$ to derive the lower bound of $p(t, x, y)$ in Theorem \ref{T:n2.1} and \ref{T:2.1} simultaneously.
\medskip

Using \cite[Theorem 5.2]{MR2524930} and Lemma~\ref{L:plbd1}, the proof of the following Proposition is similar to that of \cite[Proposition 3.3]{MR3237737}.
Thus, we omit the proof. 

\begin{prop}\label{pl:st1}
Let $D$ be an arbitrary open set and let $a$ and $T$ be positive constants.
Suppose that $(t, x, y)\in (0, T]\times D\times D$, with $\delta_D(x) \ge a  \phi^{-1}(t) \geq 2|x-y|$.
Then
there exists a positive constant $C_{\ref{pl:st1}}=C_{\ref{pl:st1}}(a, {L_0, \phi}, T)$ such that $p_D(t,x,y) \ge C_{\ref{pl:st1}}[\phi^{-1}(t)]^{-d}$.
\end{prop}

From Lemma \ref{L:plbd1} and Proposition \ref{pl:st1} we see that,  
under the condition $\bf(WS)$ on $\phi$, the behavor of $Y$ is locally stable in  terms of $\phi$.

\begin{prop}\label{pl:st2}
Let $D$ be an arbitrary open set and let $a$ and $T$ be positive constants. 
\begin{description}
\item[(1)]
Suppose that the jumping intensity kernel $J$ satisfies the conditions {\bf(J1.2)} and  {\bf(J1.3)}.
Then for every $M>0$,  
 there exists a constant $C_{\ref{pl:st2}.1}=C_{\ref{pl:st2}.1}(a,M, L_0, {\phi}, T)>0$ such that 
 for all $(t, x, y)\in (0, T]\times D\times D$ with $\delta_D(x)\wedge \delta_D (y) \ge a \phi^{-1}(t)$ and $a \phi^{-1}(t) \leq 2|x-y| {\le 2M}$ we have 
 $p_D(t, x, y)\ge C_{\ref{pl:st2}.1}t \nu(|x-y|)$.
\item[(2)] 
Suppose that the jumping intensity kernel $J$ satisfies the condition {\bf(J2)}.
Then there exists a constant $C_{\ref{pl:st2}}=C_{\ref{pl:st2}}(a, L_0, {\phi}, T)>0$ such that for every  $(t, x, y)\in (0, T]\times D\times D$, with $\delta_D(x)\wedge \delta_D (y) \ge a \phi^{-1}(t)$ and $a \phi^{-1}(t) \leq 2|x-y|$, we have  $p_D(t, x, y)\ge C_{\ref{pl:st2}}t 
\nu(|x-y|)/\chi(|x-y|)$.
\end{description}
\end{prop}

\proof
We first give the proof of (2). 
By Lemma~\ref{L:plbd1}, there exists    $c_1=c_1(a,{L_0, \phi}, T)>0$ such that 
$$\inf_{\{|z-y|\le 4^{-1}a\phi^{-1}(t)\}} \p_z(\tau_{B(z, 6^{-1}a\phi^{-1}(t))}>t)\ge c_1.$$
Thus by
the strong Markov property
$$\p_x \left( Y^D_t \in B \big( y, \,  2^{-1} a \phi^{-1}(t) \big) \right) \ge 
c_1\p_x \left(Y^D \hbox{ hits the ball } B(y, \,
4^{-1}a \phi^{-1}(t))\mbox{ by time } t \right).
$$
Using this and the L\'evy system in \eqref{e:levy}, we obtain
\begin{align}
&\p_x \left( Y^D_t \in B \big( y, \,  2^{-1} a \phi^{-1}(t) \big) \right)\nn\\  
\ge  &\, c_1\p_x(Y_{t\wedge \tau_{B(x, 6^{-1} a \phi^{-1}(t))}}^D\in B(y,\, 4^{-1}a \phi^{-1}(t))\hbox{ and }t \wedge \tau_{B(x, 6^{-1}a  \phi^{-1}(t))} \hbox{ is a jumping time })\nonumber \\
%
= &\,c_1 \E_x \left[\int_0^{t\wedge \tau_{B(x, 6 \cdot 2^{-5}a \phi^{-1}(t))}} \int_{B(y, \, 4^{-1}a \phi^{-1}(t))}
J(Y_s, u) duds \right]. \label{e:nv1}
\end{align}
Lemma~\ref{L:plbd1} also implies that 
\begin{equation}\label{eq:lowtau}
\E_x \left[ t \wedge
\tau_{B(x, 6 \cdot 2^{-5}a \phi^{-1}(t))} \right] \,\ge\,t \, \p_x
\left(\tau_{B(x, 6 \cdot 2^{-5}a \phi^{-1}(t))} \ge  t \right)
 \,\ge\, c_2\,t
\qquad \hbox{ for all } t\in (0, T].
\end{equation}

Let  $w$ be the point on the line connecting $x$ and $y$ (i.e., $|x-y|=|x-w|+|w-y|$) such that $|w-y|=7 \cdot 2^{-5} a \phi^{-1}(t)$, 
then $B(w, 2^{-5} a \phi^{-1}(t)) \subset B(y, \, 4^{-1}a \phi^{-1}(t))$.
Moreover, for every $(z,u) \in  B(x, 6 \cdot 2^{-5}a \phi^{-1}(t)) \times B(w, 2^{-5} a \phi^{-1}(t)) $, we have 
\begin{align*}
|z-u| &< 6^{-1}a \phi^{-1}(t)+ 2^{-5} a \phi^{-1}(t)+|x-w|\nn\\
& =|x-y|+(6 \cdot 2^{-5}+2^{-5}-7 \cdot 2^{-5})a \phi^{-1}(t) =|x-y|
\end{align*}
and thus $B(w, 2^{-5} a \phi^{-1}(t)) \subset \{u:|z-u|<|x-y|\}$.
Combining this result with {\bf(J2)}, \eqref{a:kappa} and \eqref{eq:lowtau}, we obtain 
\begin{align}
 &\E_x \left[\int_0^{t\wedge \tau_{B(x, 6 \cdot 2^{-5}a \phi^{-1}(t))}} \int_{B(y, \, 4^{-1}a \phi^{-1}(t))}
J(Y_s, u) duds \right] \nn\\
\ge &\,\E_x \left[\int_0^{t\wedge \tau_{B(x, 6 \cdot 2^{-5}a \phi^{-1}(t))}} \int_{B(w, 2^{-5} a \phi^{-1}(t))}
J(Y_s, u) {\bf 1}_{\{ |Y_s -u| <|x-y| \}} duds \right]\nn \\
\ge &\,L_0^{-1} \E_x \left[ t \wedge
\tau_{B(x, 6 \cdot 2^{-5}a \phi^{-1}(t))} \right]   |B(w, 2^{-5} a \phi^{-1}(t))| \nu(|x-y|)/\chi(|x-y|) \nn\\
>&\, c_3   t[\phi^{-1}(t)]^d \nu(|x-y|)/\chi(|x-y|). \label{e:nv2}
\end{align} 

Then, using the semigroup property along with  Proposition~\ref{pl:st1}, \eqref{e:nv2} and \eqref{i:sc}, the proposition follows from the proof of \cite[Proposition 3.5]{MR3237737}.

 The proof of (1) is identical to the that of (2) except that we apply  {\bf(J1.3)} in \eqref{e:nv2} instead of {\bf(J2)} and \eqref{a:kappa}.
\qed

\medskip 

Combining Propositions~\ref{pl:st1} and~\ref{pl:st2},
we obtain the following preliminary lower bound of $p_D(t,x,y)$.
Note that the lower bound in Proposition \ref{pl:st3}(1) is the sharp interior lower bound of $p_D(t,x,y)$ under the conditions {\bf(J1.2)} and  {\bf(J1.3)}.
Moreover, under the condition {\bf(J2)}, the lower bound in Proposition \ref{pl:st3}(2)
that yields the sharp interior lower bound of $p_D(t,x,y)$ for the case $\beta\in[0, 1]$ and the case $\beta\in(1, \infty]$ with $|x-y|<1$.

\begin{prop}\label{pl:st3}
Let $D$ be an arbitrary open set and let $a$ and $T$ be positive constants.
\begin{description}
\item[(1)]
Suppose that the jumping intensity kernel $J$ satisfies the conditions {\bf(J1.2)} and  {\bf(J1.3)}.
Then, for every $(t, x, y)\in (0, T]\times D\times D$ and $M>0$, with $\delta_D(x)\wedge \delta_D(y) \ge a \phi^{-1}(t)$ and $|x-y|< M $, there exists a constant $C_{\ref{pl:st3}.1}=C_{\ref{pl:st3}.1}(a, M, L_0, {\phi}, T)>0$ such that
\begin{align*}
p_D(t, x, y)\,\ge \,
C_{\ref{pl:st3}.1}\Big([\phi^{-1}(t)]^{-d}\wedge t\nu(|x-y|)\Big).
\end{align*}
\item[(2)]
Suppose that the jumping intensity kernel $J$ satisfies the condition {\bf(J2)}.
Then, for every $(t, x, y)\in (0, T]\times D\times D$, with $\delta_D(x)\wedge \delta_D(y) \ge a \phi^{-1}(t)$, there exists a constant $C_{\ref{pl:st3}}=C_{\ref{pl:st3}}(a, L_0, {\phi}, T)>0$ such that
\begin{align*}
p_D(t, x, y)\,\ge \,
C_{\ref{pl:st3}}\Big([\phi^{-1}(t)]^{-d}\wedge t \nu(|x-y|)/\chi(|x-y|)\Big).
\end{align*}
\end{description}
\end{prop}

For the remainder of this section, assume that the jumping intensity kernel $J$ satisfies the condition {\bf(J2)} for $\beta\in(1, \infty]$ with  $|x-y|\ge 1$.
Also, we assume that $D$ is an {\it connected} open set with the following property: there exist $\lambda_1 \in [1, \infty)$ and $\lambda_2 \in (0, 1]$ such that for every $r \le 1$ and  $x,y$ in the same component of $D$ with $\delta_D(x)\wedge \delta_D(y)\ge r$ there exists in $D$ a length parameterized rectifiable curve $l$ connecting $x$ to $y$ with the length $|l|$ of $l$  less than or equal to $\lambda_1|x-y|$ and $\delta_D(l(u))\geq\lambda_2 r$ for $u\in[0,|l|].$ 

Now we prove the preliminary lower bound of $p_D(t,x,y)$ separately for the case $\beta=\infty$ and the case $\beta\in(1, \infty)$. We will closely follow the proofs of \cite[Theorem 3.6]{CKK} and \cite[Theorem 5.5]{MR2806700}.

\begin{prop}\label{pl:st4}
Let $\beta=\infty$.  Suppose that $T>0$ and $a \in (0, \left(4\phi^{-1}(T)\right)^{-1}]$.
Then there exist constants $C_{\ref{pl:st4}.i}=C_{\ref{pl:st4}.i}(a, L_0, \phi, T, \lambda_1, \lambda_2)>0$, $i=1,2$, such that for any $x, y \in D$ with $\delta_D(x)\wedge\delta_D(y) \ge a \phi^{-1}(t)$, $|x-y|\ge 1$, and $t \le T$ we have  
\begin{align*}
p_D(t, x, y)\,\ge \,C_{\ref{pl:st4}.1} \left(\frac{t}{T|x-y|}\right)^{C_{\ref{pl:st4}.2}|x-y|}.
\end{align*}
\end{prop}

\proof 
Let $R_1:=|x-y| \ge 1$, and by the assumption on $D$, 
there is a length parameterized curve $l\subset D$ connecting $x$ and $y$ such that the total length $|l|\le \lambda_1 R_1$ and $\delta_D(l(u))\ge \lambda_2 a\phi^{-1}(t)$ for every $u\in [0, |l|]$.
Define $k$ be the integer satisfying $(4 \le ) 4 \lambda_1 R_1\leq k <4 \lambda_1 R_1+1\le 5 \lambda_1 R_1$ and $r_t:= 2^{-1}\lambda_2 a \phi^{-1}(t) \le 8^{-1}$.
For each $i=0,1,2,\ldots,k$, let $x_i:=l(i|l|/k)$ and  $B_i:=B(x_i, r_t)$, 
then $\delta_D(x_i)\ge 2r_t$ and $B_i\subset B(x_i, 2 r_t)\subset D$.
Since $4 \lambda_1 R_1\le k$ for each $y_i\in B_i$, we have  
\begin{eqnarray}\label{e:stst1}
|y_i-y_{i+1}|\leq |y_i-x_i|+|x_i-x_{i+1}|+|x_{i+1}-y_{i+1}| \leq 
\frac{1}{8}+\frac{|l|}{k}+ \frac{1}{8}
\le \frac{\lambda_1 R_1}{4\lambda_1 R_1}+\frac{1}{4}=\frac{1}{2} .
\end{eqnarray}
Moreover $\delta_D(y_i)\ge\delta_D(x_i)-|y_i-x_i |\ge r_t\ge r_{t/k}$.
Thus by Proposition~\ref{pl:st3}(2), there are constants $c_i=c_i(a,L_0, { \phi},T)>0$, $i=1,2$, such that for $(y_i, y_{i+1})\in B_i\times B_{i+1}$ we have
\begin{eqnarray}\label{e:stst2}
p_D(t/k, y_i, y_{i+1})\ge c_1\left(\frac{1}{[\phi^{-1}(t/k)]^d}\wedge \frac{t/k}{\phi(|y_i-y_{i+1}|)|y_i-y_{i+1}|^d}\right)\ge  c_2 \,t/(Tk).
\end{eqnarray}
The last inequality comes from $t/k\le T/4$
for the first part and \eqref{e:stst1} for the second part.
Note that $r_t\ge c_3 (t/kT)^{1/\la}$ for some $c_3=c_3( a, \phi, T, \lambda_2)$ by  \eqref{i:sc}.
Hence, combining these observations and the fact that $k\asymp R_1$, we conclude that 
\begin{align*}
&p_D(t,x,y)\ge \int_{B_1}\ldots\int_{B_{k-1}}p_D(t/k,x,y_1)\ldots p_D(t/k, y_{k-1},y) dy_{k-1}\ldots dy_1\\
&\ge (c_2t(Tk)^{-1})^k\Pi^{k-1}_{i=1} |B_i| \ge ( c_4  t(Tk)^{-1})^{c_5 k}  \ge (c_6t(TR_1)^{-1})^{c_7R_1}  \ge c_8 (t(TR_1)^{-1})^{c_9R_1}.
\end{align*}
\qed

\begin{prop}\label{pl:st5}
Let $\beta\in(1, \infty)$. 
Suppose that $T>0$ and $a \in (0, \left(4\phi^{-1}(T)\right)^{-1}]$. 
Then there exist constants  $C_{\ref{pl:st5}.i}=C_{\ref{pl:st5}.i}(a,\beta, \chi,  L_0, \phi, T,  \lambda_1, \lambda_2)>0$, $i=1,2$ such that for any $x, y \in D$ with $\delta_D(x)\wedge\delta_D(y) \ge a \phi^{-1}(t)$, $|x-y|\ge 1$, and $t \le T$ we have  
\begin{eqnarray*}
p_D(t,x,y)\ge C_{\ref{pl:st5}.1} t \exp\left\{ -C_{\ref{pl:st5}.2} \left(|x-y|\left(\log\frac{T|x-y|}{t}\right)^{\frac{\beta -1}{\beta}}\wedge (|x-y|)^{\beta}\right)\right\}.
\end{eqnarray*}
\end{prop} 

\proof
Let $R_1:=|x-y|$. 
If either $1 \le R_1 \le 2$ or $R_1(\log(TR_1/t))^{(\beta-1)/\beta}\ge (R_1)^{\beta}$, the proposition holds by virtue of Proposition~\ref{pl:st3}(2). 
Thus for the remainder of this proof we assume that $R_1> 2$ and $R_1(\log(TR_1/t))^{(\beta-1)/\beta}< (R_1)^{\beta}$, which is equivalent to  $1< R_1 \left(\log{TR_1}/{t}\right)^{-1/\beta}$ and $R_1 \exp(-R_1^{\beta})< t/T$.

Let $k\ge 2$ be a positive integer such that
\begin{align}\label{e:k}
R_1 \left(\log\frac{TR_1}{t}\right)^{-1/\beta}\le k<R_1 \left(\log\frac{TR_1}{t}\right)^{-1/\beta}+1< 2R_1 \left(\log\frac{TR_1}{t}\right)^{-1/\beta}
\end{align} 
then ${R_1}/{k}\ge {2}^{-1}(\log(TR_1/t))^{1/\beta}\ge {2}^{-1}(\log2)^{1/\beta}=:c_0.$

By the assumption on $D$, there is a length parameterized curve $l\subset D$ connecting $x$ and $y$ such that $|l|\le \lambda_1 R_1$ and $\delta_D(l(u))\ge \lambda_2 a\phi^{-1}(t)$ for every $u\in [0, |l|]$.
Let $r_t:=(\,2^{-1}{\lambda_2}a\phi^{-1}(t)\,)\,\wedge \,(c_0/2)$ and define $x_i:=l(i |l|/k)$ and $B_i:=B(x_i, r_t)$, with $i=0, 1,  \ldots, k$.
For every  $y_i \in B_i$ , $\delta_D(y_i)\ge 2^{-1} \lambda_2 a \phi^{-1}(t)>2^{-1} \lambda_2 a \phi^{-1}(t/k) $ and
\begin{align}
 |y_i-y_{i+1}|& \le|x_i-x_{i+1}|+2r_t \le \frac{|l|}{k}+c_0\le\left(\lambda_1 +1 \right)\frac{R_1}{k}.\label{e:uy}
\end{align}
By Proposition~\ref{pl:st3}(2) and \eqref{e:uy}, and using the facts that $t/k\le T/2$ and $R_1/k\ge c_0$,
we have that for any $(y_i, y_{i+1})\in B_i\times B_{i+1}$,
\begin{align*}
p_D(t/k, y_i, y_{i+1})  \ge c_{1} \left( \frac{1}{[\phi^{-1}(t/k)]^d}\wedge\frac{t}{k}\cdot  \nu(|y_i-y_{i+1}|)/\chi(|y_i-y_{i+1}|)
\right)
\ge c_2\frac{t}{k}\cdot \frac{e^{-c_{3}(R_1/k)^{\beta}}}{\phi(R_1/k)(R_1/k)^d}
\end{align*}
for some constants $c_i=c_i(a,  L_0, \phi, \chi, \beta,  T,  \lambda_1)>0$, $i=2,3$.
Since  $\phi(R_1/k)\le c_4 (R_1/k)^{\ua}$ by {\bf(WS)} with $R_1/k\ge c_0$, using \eqref{e:k}, we have that 
\begin{align}\label{p(t/k)}
 p_D(t/k, y_i, y_{i+1}) & \ge c_2 \cdot c_4 \frac{t}{TR_1}
\left(\frac{k}{R_1}\right)^{\ua+d-1}e^{-c_{3}(R_1/k)^{\beta}}\nn\\
&\ge \, c_2 \cdot c_4  \frac{t}{TR_1}\left(\log \frac{TR_1}{t}\right)^{-\frac{\ua+d-1}{\beta}}\left(\frac{t}{TR_1}\right)^{c_{3}}\ge \, c_5 \left(\frac{t}{TR_1}\right)^{c_{6}}
\end{align}
for some $c_i=c_i(a, L_0,  \phi, \chi,\beta,  T,  \lambda_1)$ , $i=5,6$.
Note that $r_t\ge c_7 (t/TR_1)^{1/\la}$ for some $c_7=c_7(a, \beta, \phi, \lambda_2)$
by \eqref{i:sc} and the fact that $t/TR_1\le 1/2$.
Combining this with \eqref{p(t/k)}, \eqref{e:k} and by the semigroup property,  we conclude that 
\begin{eqnarray*}
p_D(t, x, y)&\ge& \int_{B_1}\cdots\int_{B_{k-1}} p_D(t/k, x, y_1)\cdots p_D(t/k, y_{k-1}, y ) dy_1\cdots dy_{k-1}\\
&\ge& c_8\exp\{-c_9k\log({TR_1/t})\}\\
&\ge& c_8\exp\left\{-c_9\left(2R_1 \log\left(\frac{TR_1}{t}\right)^{-1/\beta}\right)\log\frac{TR_1}{t}\right\}\\
&\ge&c_8 \exp\left\{-2c_{9}\cdot R_1 \log\left(\frac{TR_1}{t}\right)^{1-1/\beta}\right\}.
\end{eqnarray*}
\qed

\medskip
\noindent
{\bf Proofs of the lower bounds in Theorems \ref{T:n2.1} and \ref{T:2.1}}.
The lower bound of $p(t,x,y)$ in Theorem \ref{T:n2.1} follows from  Proposition~\ref{pl:st3}(1)  with $D=\R^d$.
The lower bound of $p(t,x,y)$ in Theorem \ref{T:2.1} for the case $\beta\in[0, 1]$ and the case $\beta\in(1, \infty]$ with $|x-y|<1$  follows from Proposition \ref{pl:st3}(2) with $D=\R^d$ and the remaining cases of Theorem \ref{T:2.1} follows from Propositions~\ref{pl:st4} and~\ref{pl:st5} with $D=\R^d$. 
\qed

\section{Lower bound estimates}\label{sec:lbd}

In this section, we first obtain the boundary decay in Lemma \ref{L:lbd3} using \eqref{e:har}, Lemma \ref{L:plbd1} and Lemmas \ref{L:lbd1} and \ref{L:lbd2} below.
Using the semigroup property, and then applying Lemma \ref{L:lbd3} and the preliminary lower bound estimates in Section~\ref{sec:plbd}, 
we will derive the upper bound estimate on $p_D(t,x, y)$ with the boundary decay terms for $t\le T$ in $C^{1, 1}$  open  set $D$  with $C^{1, 1}$ characteristics $(R_0, \Lambda)$. As before, we assume that $R_0<1$ and $\Lambda>1$.

We first introduce the next  lemma  (for the proof see \cite[Lemma 3.3]{MR2923420}).

\begin{lem}\label{L:lbd1}
Suppose that $E\subset \R^d$ be an open set and $U_1, U_2\subset E$ be disjoint open subsets.
If $x\in U_1$, $y\in U_2$ and $t>0$, we have
\begin{align*}
p_E(t,x,y)\ge t\,\ \p_x(\tau_{U_1}>t)\,\ \p_y(\tau_{U_2}>t) \inf_{(u, w)\in U_1\times U_2}J(u, w).
\end{align*}
\end{lem}

For the remainder of the section, we assume that $Y$ is the symmetric pure jump Hunt process with the jumping intensity kernel  $J$ satisfying the conditions {\bf(J1.1)} , {\bf(J1.2)} and \textbf{(K$_{\eta}$)}.
For any $T>0$, let
\begin{equation*}
\wh a_{T}:=\wh a_{T,R_0}:= \frac{R_0}{80\, \phi^{-1}(T)},
\end{equation*}
and for $x\in D$ we use $z_x$ to denote a point on $\partial D$ such that  $|z_x-x|=\delta_D(x)$.

We first give the survival probability where $x$ is near the boundary of $D$ in the following lemma.

\begin{lem}\label{L:lbd2}
Let $a \le \wh a_{T}$. Then, there exists a constant $C_{\ref{L:lbd2}}=C_{\ref{L:lbd2}}(a, \phi, L_0, L_3,  \eta, \Lambda,  T)>0$
such that for every $t\le T$ and  $x\in D$ with $ \delta_{D} (x)<a\phi^{-1}(t)$ we have
\begin{align}
\p_x(\tau_{B(z_x, 10a\phi^{-1}(t))\cap D}>t/3)\ge C_{\ref{L:lbd2}}\,\ \frac{V(\delta_{D} (x))}{\sqrt{t}}.
\end{align}
\end{lem}

\proof
Without loss of generality, we assume that $z_x=0$.
Consider a coordinate system $CS:=CS_0$ such that $B(0, R_0)\cap D=\{y=(\wt{y}, y_d)\in B(0, R_0) \text{ in } CS: y_d>\varphi(\wt{y})\}$, where $\varphi$ is a $C^{1,1}$ function such that $\varphi(0)=0$, $ \nabla\varphi (0)=(0, \dots, 0)$, $\| \nabla\varphi \|_\infty \leq \Lambda$, and $| \nabla \varphi (\wt y)-\nabla \varphi (\wt w)| \leq \Lambda |\wt y-\wt w|$.
Define $\varphi_1(\wt{y}):=2\Lambda|\wt{y}|$ and $V:=\{y=(\wt{y}, y_d)\in B(0, R_0) \text{ in } CS: y_d>\varphi_1(\wt{y})\}$.
Since $\varphi_1(\wt{y}) \ge 2\Lambda|\wt{y}|^{2}$ for $y\in B(0, R_0)$, the mean value theorem yields $V\subset B(0, R_0)\cap D$.

Let $U_1:=B(0, 2a\phi^{-1}(t))\cap D$, $U_2:=B(0, 10a\phi^{-1}(t))\cap D$, and   
\begin{align}
W:=\{y=(\wt{y}, y_d)\in B(0,8 a\phi^{-1}(t))\setminus B(0, 2a\phi^{-1}(t)) \text{ in } CS: y_d>\varphi_1(\wt{y})\}\subset V.
\end{align}
Since $ \Lambda |\wt{w}|=\varphi_1(\wt{w})/2<w_d /2$ for $w\in W$, we have
\begin{align}\label{e:jk1}
 \delta_D (w)> \frac{(w_d-\varphi(\wt{w}))}{(1+\Lambda)}> \frac{(w_d-\Lambda |\wt{w}|)}{(1+\Lambda)}>\frac{ w_d}{2(1+\Lambda)}\qquad
 \text{ for  } w\in W.
\end{align}
Moreover, since $|\wt w| \le (2\Lambda)^{-1}   |w| \le \Lambda^{-1}4a\phi^{-1}(t)\le  a \phi^{-1}(t)$ for $w\in W$, 
we have
\begin{align}\label{e:jk2} 
w_d^2 =|w|^2- |\wt w|^2 \ge  (2a\phi^{-1}(t))^2- (a\phi^{-1}(t))^2 
\ge (a\phi^{-1}(t))^2\qquad
 \text{ for  } w\in W.
\end{align}
Combining \eqref{e:jk1} and \eqref{e:jk2}, we obtain $\delta_D (w)> 2^{-1} (1+\Lambda)^{-1}a\phi^{-1}(t)$ and $B(w,r_1a\phi^{-1}(t))\subset U_2$ for $w\in W$, where $r_1:=(2(1+\Lambda))^{-1}$.
By virtue of the strong Markov property, Lemma~\ref{L:plbd1}, and \eqref{e:har}, we have
\begin{align*}
&\p_x(\tau_{U_2}>t/3)\,\geq\, \p_x(\tau_{U_2}>t/3, Y_{\tau_{U_1}} \in W)\,=\, \E_x[\p_{Y_{\tau_{U_1}}}(\tau_{U_2}>t/3): Y_{\tau_{U_1}}\in W]\\
\geq &\,\,\E_x[\p_{Y_{\tau_{U_1}}}(\tau_{B(Y_{\tau_{U_1}}, r_1a\phi^{-1}(t))}>t/3): Y_{\tau_{U_1}}\in{W}] \ge  \left(\inf_{z \in \R^d} \p_z(\tau_{B(z, r_1a\phi^{-1}(t))}>t/3) \right)\p_x(Y_{\tau_{U_1}}\in{W})\\
\geq&\,\, C_{\ref{L:plbd1}}\,\p_x(Y_{\tau_{U_1}}\in{W})\geq C_{\ref{L:plbd1}}\cdot  C_{\ref{T:exit}.2} \frac{V(\delta_D(x))}{V(8a\phi^{-1}(t))} \geq c_1 \frac{V(\delta_D(x))}{\sqrt{t}}.
\end{align*}
By the subadditivity of $V$ and \eqref{e:com}, 
$V(8a\phi^{-1}(t))\le (8a+1) V(\phi^{-1}(t)) \asymp \sqrt{t}$,
and therefore we obtain the last inequality.
\qed

\medskip

We introduce the following definition for the subsequent lemma.
\begin{definition}\label{def:UB}
Let $0<\kappa_1\leq 1/2$.
We say that an open set $D$ is $\kappa_1$-fat if there is $R_1>0$ such that for all $x\in \overline{D}$ and all $r\in (0,R_1]$ there is a ball $B(A_r(x), \kappa_1 r)\subset D\cap B(x,r)$.
The pair $(R_1, \kappa_1)$ are called the characteristics of the $\kappa_1$-fat open set $D$.
\end{definition}

Note that a $C^{1, 1}$ open set $D$ with characteristics $(R_0, \Lambda)$ is a $\kappa_1$-fat set with characteristics $(R_1, \kappa_1)$ depending only on $R_0$, $\Lambda$, and $d$, and without loss of generality, we assume that $R_0 \le R_1$ (by choosing $R_0$ smaller if necessary). 
Let $A_r(x)$ is always the point $A_r(x)\in D$ in Definition~\ref{def:UB} for $D$.

Recall that the function $\Psi$ is defined in \eqref{e:dax}.

\begin{lem}\label{L:lbd3}
There exists a  constant 
$C_{\ref{L:lbd3}}=C_{\ref{L:lbd3}}(\phi, L_0, L_3,  \eta, R_0, \Lambda,  T)>0$
such that, for every $t\le T $ and $x\in D$, we can find $x_1$ with $\delta_D(x_1)\ge  2^{-1}\kappa_1 \wh a_{T}\phi^{-1}(t)$ and $|x_1-x|  \le  6\wh a_{T}\phi^{-1}(t)$ such that 
\begin{align}
\int_{B(x_1, 4^{-1}\kappa_1  
 \wh a_{T}\phi^{-1}(t))} p_D(t/3, x, z)dz \ge C_{\ref{L:lbd3}} \Psi(t,x). 
\end{align}
\end{lem}

\proof
Let $r_t:= \wh a_{T}\phi^{-1}(t)\le R_0/80\le 1/80$
and we consider the case $\delta_D(x)< 2^{-1}\kappa_1 r_t$ first. 
In this case we let  $x_1:=A_{6r_t}(z_x)$ and 
denote $B_{x_1}:=B(x_1, 4^{-1}\kappa_1   r_t)$ and $B_{z_x}:=B(z_x, 5\kappa_1 r_t)\cap D$ so that $B_{x_1}\cap B_{z_x} =\emptyset$.
For any  $u \in B_{z_x}$ and $w\in B_{x_1} $, 
$$|u-w| \le |u-z_x|+|z_x-x_1|+|x_1-w|\le 12 \kappa_1 r_t\le 1.$$
Since $\phi(12\kappa_1 r_t)\asymp \phi(\phi^{-1}(t))=t$ by {\bf(WS)},
using  {\bf(J1.1)}, \eqref{d:nu} and \eqref{a:kappa},  we have that
$$\inf_{(u, w)\in B_{z_x}\times B_{x_1}} J(u, w)\ge  L_0^{-1} \phi(12\kappa_1 r_t)^{-1}|12\kappa_1 r_t|^{-d}\ge c_1 t^{-1}[\phi^{-1}(t)]^{-d}$$
for some constant $c_1:=c_1(\phi, L_0, R_0, \Lambda, T)>0$.
Therefore, Lemmas~\ref{L:lbd1}, \ref{L:lbd2}, and \ref{L:plbd1} implies that
\begin{eqnarray*}
\int_{B_{x_1}} p_D(t/3,x,z)dz 
&\ge& \frac{t}{3} \int_{B_{x_1}} \p_x(\tau_{B_{z_x}}>t/3)\,\,\p_z(\tau_{B_{x_1}}>t/3)  \cdot \inf_{(u, w)\in B_{z_x}\times B_{x_1}}J(u, w) dz\\
&\ge& \frac{1}{3}\,\ \p_x(\tau_{B_{z_x}}>t/3)\cdot C_{\ref{L:plbd1}} \int_{B_{x_1}} dz\cdot c_1 \frac{1}{[\phi^{-1}(t)]^{d}}\\
&=&c_2 \p_x(\tau_{B_{z_x}}>t/3)\,\ge\, c_2 \cdot C_{\ref{L:lbd2}}\,\ \frac{V(\delta_D(x))}{\sqrt{t}}.
\end{eqnarray*}

For $\delta_D(x)\ge 2^{-1}\kappa_1 r_t$, let $x_1=x$ and $B_{x_1}:=B(x_1,  4^{-1}\kappa_1  r_t)$.
By Lemma~\ref{L:plbd1}, 
\begin{eqnarray*}
\int_{B_{x_1}} p_D(t/3,x,z)dz 
\ge \int_{B_{x_1}} p_{B_{x_1}}(t/3, x, z) dz= \p_x(\tau_{B_{x_1}}>t/3)>C_{\ref{L:plbd1}},
\end{eqnarray*}
and this proves the lemma.
\qed

We are now ready to give the proof of the lower bound estimates for $p_D(t, x, y)$.
Recall our assumption that $D$ is a $C^{1, 1}$ open set. 
When the jumping intensity $J$ of $Y$ satisfies {\bf(J2)},
for the cases $\beta\in(1,\infty)$ with $|x-y| \ge 1$ and $\beta=\infty$ with $|x-y| \ge 4/5$, 
we assume in addition that the path distance in each connected component of $D$ is comparable to the Euclidean distance with characteristic $\lambda_1$. 
Note that combining this assumption with $C^{1, 1}$ assumption entails that $D$ satisfies the assumption made before Proposition~\ref{pl:st4}.  

\medskip

\noindent {\bf Proofs of the lower bound of $p_D(t, x, y)$ in  Theorems \ref{t:nmain}(1), \ref{t:main}(2) and \ref{t:main}(3)}. 
Let $r_t:= \wh a_{T}\phi^{-1}(t)\le R_0/80\le 1/80$. By Lemma~\ref{L:lbd3}, for any $x, y\in D$, there exists $x_1, y_1\in D$ such that $\delta_D(x_1) \wedge \delta_D(y_1) \ge 2^{-1}\kappa_1 r_t$ and $|x_1-x|  \vee |y_1-y| \le  6r_t$, and
\begin{align}\label{eq:m11}
\int_{B_{x_1}} p_D(t/3, x, z)dz\int_{B_{y_1}} p_D(t/3, y, z)dz \ge C_{\ref{L:lbd3}}^2 \Psi(t,x) \Psi(t,y),
\end{align}
where $B_{x_1}:=B(x_1, 4^{-1}\kappa_1   r_t)$ and $B_{y_1}:=B(y_1, 4^{-1}\kappa_1   r_t)$.
Thus by the semigroup property,
\begin{align}\label{eq:m1}
p_D(t, x, y)
&=\int_D \int_D p_D(t/3, x, u) p_D(t/3, u, w) p_D(t/3,w,y)du dw\nonumber\\ 
&\ge \int_{B_{x_1}} p_D(t/3,x,u)du \int_{B_{y_1}} p_D(t/3, y, w) dw  \left(\inf_{(u,w)\in B_{x_1}\times B_{y_1}} p_D(t/3, u, w) \right) \nonumber\\
&\ge  C_{\ref{L:lbd3}}^2\Psi(t,x) \Psi(t,y) \inf_{(u,w)\in B_{x_1}\times B_{y_1}}  p_D(t/3, u, w) .
\end{align}

We now carefully calculate the lower bounds of $p_D(t/3,u,w)$ on $B_{x_1} \times B_{y_1}$.
Since $|x-x_1|\vee |y-y_1|\le 6r_t$, for $u\in B_{x_1}$ and $w\in B_{y_1}$ we have
\begin{align}\label{e:dfdf5}
&|x-y|-6^{-1} \le |x-y|-(12+ (\kappa_1/2))r_t \nn\\
&\le |u-w| \le |x-y|+(12+ (\kappa_1/2))r_t \le  |x-y|+ 6^{-1}
\end{align}
and 
$\delta_D(u)\wedge \delta_D(w)\ge 4^{-1}\kappa_1  r_t$. 

We first assume that the jumping intensity kernel $J$ satisfies the condition {\bf(J2)}.
Let $\beta\in[0,1]$.
If $|x-y| \le 15 r_t$, then $|u-w|\le 28 r_t<1$ and 
$\phi(|u-w|)|u-w|^d\le c_1 t[\phi^{-1}(t)]^{d}$ 
since $\phi(|u-w|)\le \phi(28\kappa_1 r_t)\asymp \phi(\phi^{-1}(t))=t$ by {\bf(WS)}.
If $|x-y| >15 r_t$, then  $|u-w|\le |x-y|+6^{-1}$ and  
$\phi(|u-w|)|u-w|^d\le c_2\phi(|x-y|)|x-y|^d$  
since $r\to \phi(r)$ is increasing and using {\bf(WS)}.
Combining these observations with  Proposition \ref{pl:st3}(2), \eqref{d:nu}  and  \eqref{e:Exp}, 
\begin{align*}
p_D(t/3, u, w)\ge c_3
\left([\phi^{-1}(t)]^{-d}\wedge t e^{-\gamma_2|u-w|^\beta}
\nu(|u-w|)
\right)\ge c_4 \left([\phi^{-1}(t)]^{-d}\wedge 
t e^{-\gamma_2|x-y|^\beta} 
\nu(|x-y|)\right).
\end{align*}
If $\beta\in(1, \infty]$ and  $|x-y| \le 4/5$, then \eqref{e:dfdf5} yields $|u-w|\le|x-y|+6^{-1} <1$. 
Similar to the above case, considering the cases $|x-y|\le 15 r_t$ and $|x-y| >15 r_t$ separately, we have $p_D(t/3, u, w)\ge c_5 \left([\phi^{-1}(t)]^{-d}\wedge t  \cdot \nu(|x-y|)\right)$.  Moreover, 
\begin{itemize}
\item[(1)]
if $\beta\in(1, \infty)$ and $4/5 \le |x-y|<2$, then $|u-w|\asymp 1$.
Thus by Proposition~\ref{pl:st3}(2), we have $p_D(t/3, u, w)\ge c_6 t$.
\end{itemize}
Hence combining \eqref{eq:m1} with these observations, we have proved the lower bound of $p_D(t, x, y)$ in Theorem~\ref{t:main}(2).

Suppose the jumping intensity kernel $J$ satisfies the condition {\bf(J1)} and $M>0$. 
Let $|x-y|<M$.
Similar to the $\beta\in[0,1]$ case, applying Proposition \ref{pl:st3}(1) instead of Proposition \ref{pl:st3}(2)
and considering $|x-y|\le 15 r_t\wedge M$ and $15 r_t\wedge M< |x-y| \le M$ separately, we have $p_D(t/3, u, w)\ge c_{7} \left([\phi^{-1}(t)]^{-d}\wedge t  \cdot \nu(|x-y|)\right)$.
Hence combining \eqref{eq:m1} with this, we have proved the lower bound of $p_D(t, x, y)$ in Theorem~\ref{t:nmain}(1).

We now return to the assumption that the jumping intensity kernel $J$ satisfies the condition {\bf(J2)}, further assume that
 the path distance in  $D$ is comparable to the Euclidean distance.
If  $4/5\le |x-y|$, then \eqref{e:dfdf5} yields $|u-w| \asymp |x-y|$. Recall that we have already discuss the case $\beta\in(1, \infty)$ and $4/5 \le |x-y|<2$ in (1). 
We now consider $p_D(t/3, u, w)$ in each of the remaining cases. 
\begin{itemize}
\item[(2)]
If $\beta=\infty$ and $4/5\le |x-y|<2$, then by Propositions~\ref{pl:st3} and~\ref{pl:st4}, we have
\begin{align*}
p_D(t/3, u,w)\ge c_8 \frac{4t }{5T|x-y|}\ge c_8\left(\frac{4t}{5T|x-y|}\right)^{5|x-y|/4}.
\end{align*}

\item[(3)]
If $\beta\in(1, \infty)$ and $2\le |x-y|$, then $1 <|u-w|$ and from Proposition~\ref{pl:st5} and \eqref{e:dfdf5} we obtain
\begin{align*}
&p_D(t/3, u, w)\ge c_9  t\exp\left\{-c_{10}\left(|u-w|\left(\log\frac{T|u-w|}{t}\right)^{\frac{\beta-1}{\beta}}\wedge |u-w|^{\beta}\right)\right\}\\
&\ge c_9  t\exp\left\{-c_{10}\left((5|x-y|/4)\left(\log\left(\frac{T(|x-y|+6^{-1})}{t}\right)\right)^{\frac{\beta-1}{\beta}}\wedge (5|x-y|/4)^{\beta}\right)\right\}\\
&\ge c_9  t\exp\left\{-c_{11}\left(|x-y|\left(\log\frac{T|x-y|}{t}\right)^{\frac{\beta-1}{\beta}}\wedge |x-y|^{\beta}\right)\right\}.
\end{align*}
The last inequality comes from the inequality $\log r \le \log(r+b)\le 2\log r$ for $r\ge 2\vee b>0$.  

\item[(4)]
If $\beta=\infty$ and $2\le |x-y|$, then $1<|u-w|$ and from Proposition~\ref{pl:st4} and \eqref{e:dfdf5} we have 
\begin{align*}
&p_D(t/3, u, w)\ge c_{12}\left(\frac{t}{T|u-w|}\right)^{c_{13} |u-w|}
\ge c_{12}\left(\frac{t}{T(|x-y|+6^{-1})}\right)^{c_{13}5|x-y|/4}\\
&\ge c_{12}\left(\frac{t}{T|x-y|}\right)^{c_{13}5|x-y|/2}\ge c_{12}\left(\frac{4t}{5T|x-y|}\right)^{c_{13}5|x-y|/2}.
\end{align*}
The second last inequality holds by virtue of the inequality $r^2\ge r+b$ for $r\ge 2\vee b>0$.
\end{itemize}
Hence combining \eqref{eq:m1} with the above observations $(1)-(4)$ on the lower bound of $p_D(t/3, u, w)$, we have proved the lower bound of $p_D(t, x, y)$ in Theorem \ref{t:main}(3).
\qed

\medskip

\noindent {\bf Proof of Theorem \ref{t:main}(4)}. 
Let $D(x)$ and $D(y)$ be connected components containing $x$ and $y$, respectively.
By definition of a $C^{1, 1}$ open set, the distance between $x$ and $y$ is at least $R_0$.
Using Lemma~\ref{L:lbd3}, we find that $x_1\in D(x)$ and $y_1\in D(y)$.
Define $B_{x_1}$ and $B_{y_1}$ in the same way as when beginning the proof of Theorem~\ref{t:main}(2) and~\ref{t:main}(3) so that \eqref{eq:m11} holds. 

For any $u\in B_{x_1}$ and $w\in B_{y_1}$, since $3R_0/4 \le 3|x-y|/4\le  |u-w|\le 5|x-y|/4$, by Proposition~\ref{pl:st3} (2) and \eqref{i:sc},
\begin{align*}
p_D(t/3, u, w)\ge  c_{1}
 t \nu(|u-w|)e^{-\gamma_2 |u-w|^{\beta}}\ge c_2
t\nu(|x-y|)e^{-\gamma_2(5|x-y|/4)^{\beta}}.
\end{align*} 
By the semigroup property, combining \eqref{eq:m11} and this observation,  we conclude that
\begin{align*}
p_D(t, x, y)
&\ge\int_{B_{x_1}} \int_{B_{y_1}} p_D(t/3, x, w) p_D(t/3, u, w) p_D(t/3,w,y)dw dv\nonumber\\ 
&\ge \int_{B_{x_1}} p_D(t/3,x,u)du \int_{B_{y_1}} p_D(t/3, y, w) dw \cdot \inf_{(u, w)\in B_{x_1} \times B_{y_1}} p_D(t/3, u, w) \nonumber\\
&\ge c_3 \Psi(t,x) \Psi(t,y) \cdot t\nu(|x-y|)e^{-\gamma_2(5|x-y|/4)^{\beta}}.
\end{align*}
\qed

\noindent {\bf Proofs of Theorems \ref{t:nmain}(2) and \ref{t:main}(5)}. 
Using Lemmas \ref{L:exit2} and \ref{L:lbd2} instead of \cite[(5.1) and (5.10)]{MR3237737}, and by the fact that \eqref{e:com} and $D$ is bounded (and connected 
when  $J$ satisfies the condition {\bf {\bf(J2)}} and $\beta=\infty$), we can  obtain the large time heat kernel estimates for $p_D(t, x, y)$ following the  proofs of \cite[Theorems 1.3(iii) and 1.5(iii)]{MR3237737}, so we omit the proofs.

\section{Green function and boundary Harnack inequality}\label{sec:bhi}

In this section we give the Green function estimates and establish the boundary Harnack inequality as applications of the Dirichlet heat kernel estimates.
\medskip

\noindent {\bf Proof of Theorem \ref{C:green}.}
When $d\ge2$, the proof of Green function estimates is almost identical to the one of \cite[Section 7]{MR3237737}.  Thus we skip the proof.

Suppose $d=1$.   Note that by the inequality in Proposition \ref{VbisEst},  we have \begin{align}\label{n:Vbis}
V^{'}(r)\le c \frac{V(r)}{r}\qquad \mbox{ for } 0<r\le M,
\end{align}
Using \eqref{n:Vbis} instead of \cite[(7.3)]{MR3237737},
one can obtain the Green function estimates by following the proofs in \cite[Section 7]{MR3237737} line by line.
Indeed, for any $T>0$, let
 $$K_T(a, r):= a+\phi(r)\int_{\phi(r)/T}^1 \left(1\wedge \frac{ua}{\phi(r)}\right)\frac{1}{u^2\phi^{-1}(u^{-1}\phi(r))} du+\frac{\phi(r)}{r}\left(1\wedge \frac{a}{\phi(r)}\right)$$
which is defined in \cite[(7.4)]{MR3237737}.
By the same  proof of \cite[Theorem 7.3(iii)]{MR3237737}, we have that 
$$G_D(x, y)\asymp K_{T_1}(a(x, y), |x-y|)$$ 
where $a(x, y)=\sqrt{\phi(\delta_D(x))}\sqrt{\phi(\delta_D(y))}$.
Recall that $C_I$ is the constant in \eqref{i:sc}. Let $T_1:=(2\vee (2C_I)^\ua)\phi(\diam(D))$. 
Since $0<a(x, y)\le\phi(\diam(D))=  (2^{-1}\wedge (2C_I)^{-\ua})T_1$ and $\phi(|x-y|)\le \phi(\diam (D))\le T_1/2$,
it is enough to show that for any $T>0$ and for any $0<a\le (2^{-1}\wedge (2C_I)^{-\ua})T$ and $0<\phi(r)\le T/2$,
\begin{align}\label{n:hT}
K_T(a, r)\asymp \frac{a}{r}\wedge\left(\frac{a}{\phi^{-1}(a)}+\left(\int_r^{\phi^{-1}(a)}\frac{\phi(s)}{s^2}\right)^+\right)
\end{align}
where $x^+:=x\vee 0$.

When $0<a<\phi(r)\le T/2$,  the proof of \eqref{n:hT} is the same as that of \cite[Lemma 7.2]{MR3237737}. 
Now we assume that $\phi(r)\le a\le (2^{-1}\wedge (2C_I)^{-\ua})T$.
Using \eqref{e:com}, we have  $c_1^{-1} V(r)^2 \le \phi(r) \le c_1V(r)^2$
for some constant $c_1>1$. 
Thus by the change of variable $u=V(r)^2/V(s)^2$, we have that
\begin{align*}
\int_{\phi(r)/T}^{\phi(r)/a}\frac{a du}{u\phi^{-1}(u^{-1}\phi(r))} 
\le \int_{V(r)^2/(c_1T)}^{c_1V(r)^2/a}\frac {a du}{u\phi^{-1}(u^{-1}\phi(r))} 
\le \int_{1/V(a/c_1)^{2}}^{1/V(c_1T)^{2}}\frac{2a}{\phi^{-1}(c_1^{-1}V(s)^2)}\frac{V^{'}(s)}{V(s)}ds.
\end{align*}
Since $\phi^{-1}(c_1^{-1}V(s)^2)\ge \phi^{-1}(c_1^{-2}\phi(s))\ge c_2s$ by \eqref{e:com} and \eqref{i:sc}, combining this with  \eqref{n:Vbis},
we have that 
\begin{align}\label{e:new}
\int_{\phi(r)/T}^{\phi(r)/a}\frac{a \cdot du}{u\phi^{-1}(u^{-1}\phi(r))} 
\le c_3 a \int_{1/V(a/c_1)^{2}}^{1/V(c_1T)^{2}} \frac{1}{s^2}ds\le c_4 \frac{ a}{\phi^{-1}(a)}.
\end{align}
For the last inequality, we again used \eqref{e:com} and \eqref{i:sc}.
Applying \eqref{e:new} to the proof of the upper bound for $K_T(a, r)$ in \cite[(7.6)]{MR3237737}, and following  the rest of the proof of \cite[Theorem 7.3(iii)]{MR3237737} for the $\phi(r)\le a\le (2^{-1}\wedge (2C_I)^{-\ua})T$ case, we obtain \eqref{n:hT} and hence we prove Theorem \ref{C:green} for all dimension.
 \qed

{To prove Theorem \ref{C:BHI} we use the above estimates of Green function and the following the scale and translate invariant boundary Harnack inequality.
\begin{prop}\label{prop:BHI}
 Suppose that $D$ is an open set in $\R^d$. 
Let  $Y$ be a symmetric pure jump Hunt process whose jumping intensity kernel $J$ satisfies the conditions {\bf(J1)}, {\bf (L)}, {\bf (C)} and {\bf(K$_\eta$)}.
Then,  there exists $c=c(\phi,  \eta,  L_0, L_3, d)$ such that
for any $0<r<1$, $z\in \partial D$ and any non-negative functions $f,g$ in $\R^d$ which are regular harmonic in $D\cap B(z, r)$ with respect to $Y$, and vanish in $D^c\cap B(z, r)$, we have 
\begin{align*}
\frac{f(x)}{f(y)}\le c \frac{g(x)}{g(y)}\qquad \mbox{ for any}\,\,\, x, y\in D\cap B(z,2r/3).
\end{align*}
\end{prop}}
\begin{proof}
{To prove the claim we use \cite{MR3271268}. We only have to check assumptions stated therein.} 
Note that {\bf (C)} is an uniform version of 
\cite[{\bf Assumption C}]{MR3271268}. Thus there is a constant $c_{(2.7)}$ in \cite{MR3271268} satisfies $c_{(2.7)}(x_0,R_1,R_2)=C^*(\phi,d,R_1/R_2)$ for any $x_0\in \Rd$ and $0<R_1<R_2\leq 2$. Let $0<r<1$ and $2/3<a \le 2$. 

We first check the bounds on  the constants $c_{(2.8)}$ and $c_{(2.9)}$ in \cite{MR3271268}. 
In our case,  the constants $c_{(2.8)}$ and $c_{(2.9)}$ in \cite{MR3271268} can be taken as 
$$
c_{(2.8)}(x_0, ar, 2r) :=\inf_{\{ar\le |x_0-y|\le 2r\}}J(x_0, y)\qquad 
\mbox{ and }\qquad 
c_{(2.9)}(x_0, r)  \le C^*\left(\int_{\Rd\setminus B(x_0, 2r)}J(x_0, y)dy\right)^{-1}
$$
where $C^*=C^*(\phi, d, 1/2)$ is the constant in  {\bf (C)}.
(see \cite[(2.8) and (2.9)]{MR3271268} and  the last display of \cite[Proposition 2.9]{MR3271268}).
Then by  {\bf(J1.3)} and  {\bf(WS)}  for $c_{(2.8)}$,  and by {\bf(J1.1)}, {\bf(J1.2)} \eqref{a:kappa} and {\bf(WS)} for  $c_{(2.9)}$,  we have that
\begin{align}\label{a:C1}
c_{(2.8)}(x_0, ar, 2r) \ge c_1 \phi(r)^{-1}r^{-d}\qquad \mbox {and }\qquad c_{(2.9)}(x_0, r) \le c_2\phi(r)
\end{align}
where the constant $c_1>0$  depends on $\phi, a$ and $d$, and the constant 
$c_2>0$  depends on $\phi, L_0$ and $d$.

We now check \cite[{\bf Assumption A--D}]{MR3271268} (and its scale and translate invariant version) holds. 
First of all, since $p(t,x,y)$ is continuous, clearly the transition operators $T_t$ of $Y$ is  strong Feller.
Recall that we assume  that $T_t$ is  Feller , that is, $T_t$ maps $C_0(\R^d)$ into $C_0(\R^d)$.
Since $Y$ is symmetric,  \cite[{\bf Assumption A}]{MR3271268} holds.

Let $\widehat{A}$ be the corresponding generator on $C_0(\R^d)$ defined as
\begin{align*}
\widehat{A}u&:=\lim_{t\to 0}\frac{T_t u-u}{t}\qquad\mbox{(strong limit)}\mbox{ and }\\
D(\widehat{A})&:=\{u\in C_0(\R^d):\widehat{A}u<\infty\}.
\end{align*}
Recall the operator 
$\gener g(x)=P.V.\int(g(y)-g(x))J(x,y)dy$
defined in \eqref{e:opgen}.
Then 
\begin{align}\label{stA=gen}
C_c^2(\R^d)\subset D(\widehat{A})\qquad  \mbox{ and} \qquad \widehat{A}u=\gener u \qquad \mbox{ for any } u\in C_c^2(\R^d).
\end{align}
Indeed,  we first obtain that
for any $u\in C_c^2(\R^d)$, $\gener u\in C_0(\R^d)$ by {\bf(L)} and so,
\begin{align}\label{limfel}
\|T_t(\gener u)-\gener u\|_{\infty}\to 0\qquad \mbox{ as } t\to 0.
\end{align}
Since, from Lemma \ref{Dynkin_formula}, $M^u_t=u(Y_t)-u(Y_0)-\int_0^
{t}  {\gener } u(Y_s) ds$ is $\p_x$-martingale with respect to the filtration of $Y$, we have that
\begin{align*}
\frac{T_t u(x)-u(x)}{t}=\frac{1}{t}\,\E_x\[\int_0^t \gener u(Y_s)ds\].
\end{align*}
Thus  we obtain that for any $u\in C_c^2(\R^d)$,
\begin{align*}
\sup_x\left|\frac{T_tu(x)-u(x)}{t}-\gener u(x)\right|&=\sup_x \left|\frac{1}{t}\int_0^t  T_s\gener u(x)-\gener u(x)ds\right|\le \frac{1}{t}\int_0^t\|T_t(\gener u)-\gener u\|_{\infty}ds,
\end{align*} 
and combining this with \eqref{limfel}, we conclude \eqref{stA=gen}.
Therefore, \cite[ {\bf Assumption B}]{MR3271268} holds with ${\cal D}=C_c^2(\R^d)$.

For $0<R_1<R_2$, let $A(x,R_1,R_2)=\{y\in \R^d: R_1 <|x-y|<R_2\}$ be the open annulus around $x$, and $\overline{A}(x,R_1,R_2)$ the closure of $A(x,R_1,R_2)$.
For every compact set $K$ and open set $U$ satisfying $K\subset U\subset\R^d$, let
$${\cal F}_{K, U}:=\{f\in C_c^2(\R^d): f\equiv 1 \mbox{ in } K,\,\, f\equiv0 \mbox{ in } U^c,\,\,\mbox{ and } 0\le f(x)\le 1\},$$
and 
$\varrho(K, U):=\inf_{f\in {\cal F}_{K, U}}\sup_x\gener f(x)$.
Then by Lemma \ref{genBound} {and {\bf(WS)}}, for $2/3<a <b\leq1$ there exist $c_3=c_3(\phi,\eta, L_0,L_3, a, b)$ such that 
 for any $x_0\in \R^d$ and $0<r<1$, \begin{align}\label{a:B}
\wh \varrho(x_0, ar,br):=\varrho(\overline{A}(x_0, ar, br), A(x_0,2r/3, 2r))
+ \varrho(\overline{B}(x_0, ar), {B}(x_0, br))
\le c_3 \phi(r)^{-1}.
\end{align}

Let $B_u:=B(x_0, u)$ be a ball centered at $x_0$ with radius $u>0$.
{Let $d\ge 1$, $0<r<1$ and $x, y\in B_r$. By Theorem \ref{T:n2.1} (with $M=2$ and $T=\phi(2)$) and the semigroup property we have, for $t_0=\phi(|x-y|)$,
\begin{align}
G_{B_r}(x,y)&\leq \int^{t_0}_0 p(s,x,y)ds+\int^\infty_0 p_{B_r}(s+t_0,x,y)ds\nonumber\\&\leq C_{ \ref{T:n2.1}}\int^{t_0}_0 s \nu(|x-y|)ds +\int_{B_r}p(t_0,z,y)G_{B_r}(x,z)dz\nonumber\\ 
&\leq C_{ \ref{T:n2.1}}\left( t_0^2 \nu(|x-y|)+[\phi^{-1}(t_0)]^{-d}\E_x\tau_{B_r}\right)\nonumber\\
&= \frac{C_{ \ref{T:n2.1}}}{|x-y|^d}\left( \phi(|x-y|)+\E_x\tau_{B_r}\right).
\end{align}
Let $5/6<a<1$. For $x\in B_{5r/6}$ and $y\in B_{r}\setminus B_{ar}$,  $(a-5/6)r\le |x-y|\le 2r$.
Hence, by \eqref{a:C1} and {\bf(WS)}  we obtain 
\begin{align}\label{a:D}
c_{(2.10)}(x_0, 5r/6, ar,r):=
\sup_{\substack
{x\in B_{5r/6}\\
y\in B_r\setminus B_{ar}}}
G_{B_{r}}(x, y)
\le 
c_4\frac{\phi(r)}{r^{d}}.
\end{align}
}
where the constant $c_4$ depends on $\phi, a$ and $d$.
Hence  \cite[{\bf Assumption D}]{MR3271268} holds.

We have observed that \cite[{\bf Assumption A}--{\bf Assumption D}]{MR3271268} hold.
In addition,  by \eqref{a:C1}, \eqref{a:B} and \eqref{a:D},  
the upper bound of the constants $c_{(3.9)}$, $c_{(3.11)}$ and $c_{(1.1)}$ in \cite{MR3271268}
from 
 the expressions of the constants $c_{(3.9)}$, $c_{(3.11)}$ and $c_{(1.1)}$ in \cite[(3.9)--(3.11)]{MR3271268} so that for any $x_0\in \Rd$ and $0<r<1$,
\begin{align*}
&
c_{(3.9)}(x_0, 5r/6, 11r/12,r)\le c_6\frac{\phi(r)}{r^d},\nn\\
&c_{(3.11)}(x_0, 5r/6,r) \le 2 c_{(3.9)}(x_0, 5r/6, 11r/12,r) \\
&\qquad \cdot  \max\left(\frac{\wh \varrho 
(x_0, 11r/12, r)
}{c_{(2.8)} (x_0, 11r/12,2r) } ,\, |B(0,1)| C^* (\phi, d, 1/2) r^d\right)\le c_{7}\phi(r),\qquad\mbox{and} \nn\\
&c_{(1.1)}(x_0,2 r/3, r) \le\left(\wh \varrho(x_0, 3r/4,5r/6) \cdot  c_{(3.11)}(x_0, 5r/6,r) +C^* (\phi, d, 9/10)\right)^4\le c_{8}
\end{align*}
where  the  constants  $c_i$, $i=6,7,8$ are depending only on $\phi, \eta, L_0,L_3$ and $d$.
Therefore, we obtain  the scaling and translation invariant version of 
\cite[(BHI)] {MR3271268} for $r < 1$, with the constant $c_{(1.1)}=c_{(1.1)} (x_0, 2r/3, r)$ which is independent of $r <1$ and 
$x_0 \in \R^d$. 
\end{proof}
Alternatively, one can check the conditions in \cite[Section 4]{2015arXiv:1510.04569}, which also provides
 \cite[Corollary 4.2]{2015arXiv:1510.04569},  
 the scaling and translation invariant version of (BHI).

{We now use the above proposition to prove Theorem \ref{C:BHI}.}

\begin{proof}[ Proof of Theorem \ref{C:BHI}.]

Suppose that $D$ is a $C^{1,1}$ open set in $\R^d$ with characteristics $( R_0, \Lambda)$. 
Since  $D$ is a $C^{1,1}$ open set, it is easy to see that for any $z\in \partial D$  there exits  a bounded $C^{1,1}$ open set $U$ in $\R^d$ whose characteristics depend only on $R_0$ and $\Lambda$ (independent of $z \in \partial D$) such that 
 $B(z, 7R_0/8) \cap D \subset U  \subset B(z, R_0) \cap D$ ({if $d=1$ we can take $U=(z,z+R_0)$ or $U=(z-R_0,z)$)}. 
{Choose  a point $z_0 \in U \setminus \overline{B(z, 3R_0/4)}$ and let
$g_1(x)=G_{U}(x, z_0)$. Since $g_1$ is regular harmonic in $D\cap B(z,3R_0/4)$,  {applying} Proposition \ref{prop:BHI} 
we obtain
\begin{align*}
\frac{f(x)}{f(y)}\le c_{1}\frac{g_1(x)}{g_1(y)},\quad x, y\in D\cap B(z, r/2).
\end{align*}
Theorem \ref{C:green} implies the claim of the theorem.} 
\end{proof}

\bibliographystyle{abbrv}
\bibliography{DhkMp}

\vskip 0.3truein

{\bf Tomasz Grzywny}

Department of Pure and Applied Mathematics, 
Wroc\l{}aw University of Science and Technology,
Wybrze\.{z}e Wyspia\'{n}skiego 27, 
50-370 Wroc\l{}aw, Poland

E-mail: \texttt{tomasz.grzywny@pwr.edu.pl}

\bigskip

{\bf Kyung-Youn Kim}

 Institute of Mathematics, Academia Sinica, 6F, Astronomy-Mathematics Building, No. 1, Sec. 4, Roosevelt Road, Taipei 10617, TAIWAN

E-mail: \texttt{kykim@gate.sinica.edu.tw}

\bigskip

{\bf Panki Kim}

Department of Mathematical Sciences and Research Institute of Mathematics,
Seoul National University,
Building 27, 1 Gwanak-ro, Gwanak-gu,
Seoul 151-747, Republic of Korea

E-mail: \texttt{pkim@snu.ac.kr}

\end{document}